\documentclass[envcountsect,envcountsame]{llncs} 

\usepackage{amsmath}
\usepackage{amsfonts,amssymb}
\usepackage{color,enumerate,array}

\usepackage{graphicx} 
\usepackage{psfrag} 
\usepackage{epsf} 
\usepackage[utf8]{inputenc} 
\usepackage{epstopdf}
\usepackage{mhchem} 
\usepackage{subfigure}

\newcommand{\DEF}[1]{{\em #1\/}} 
 
\newcommand{\dset}[2]{\left\{#1 \:|\: #2\right\}}

\newcommand{\crn}{\operatorname{cr}} 
 
\newcommand\tcrn{{\operatorname{tcr}}}

\newcommand{\cA}{{\cal A}} 
\newcommand{\cB}{{\cal B}} 
\newcommand{\cC}{{\cal C}} 
\newcommand{\cG}{{\cal G}} 
\newcommand{\cK}{{\cal K}}

\newcommand{\cP}{{\cal P}}

\newcommand{\cT}{{\cal T}} 
\newcommand{\cU}{{\cal U}} 
\newcommand{\cS}{{\cal S}}

\newcommand{\cF}{{\cal F}} 
\newcommand{\cW}{{\cal W}}

\newcommand{\NN}{\mathbb N} 
\newcommand{\QQ}{\mathbb Q} 
\newcommand{\RR}{\mathbb R}

\newcommand{\ZZ}{\mathbb Z}

\title{On Degree Properties of Crossing-Critical\\Families of Graphs%
  \thanks{A short preliminary version of this paper has appeared at Graph
	Drawing 2015.}}

\author{
Drago Bokal\inst1\,%
\thanks{This research was supported by the internationalisation of Slovene higher education
within the framework of the Operational Programme for Human Resources Development 2007--2013 
and by the Slovenian Research Agency Research projects J1--8130, N1--0057, and L7--5459, 
and by the research programme P1--0297.}
	 \and
Mojca Bra\v ci\v c\inst1
	 \and 
Marek Der\v n\'ar\inst2$\,^{\star\star\star}$
	 \and 
Petr Hlin\v{e}n\'{y}\inst2\,%
\thanks{This research was supported by the Czech Science Foundation under the project 17-00837S.}
  }
  \institute{
    Faculty of Natural Sciences and Mathematics, University of Maribor,
    Slovenia, \\
    \email{drago.bokal@um.si, mojca.bracic@student.um.si},
	 \smallskip\and
    Faculty of Informatics, Masaryk University, 
    Brno, Czech Republic, \\
    \email{m.dernar@gmail.com, hlineny@fi.muni.cz}\,. 
  }

\begin{document}

\maketitle

\begin{abstract} 
Answering an open question from 2007, we construct
infinite $k$-crossing-critical families of graphs that contain
vertices of any prescribed odd degree, for any sufficiently large~$k$.
To answer this question, we introduce several properties of infinite
families of graphs and operations on the families allowing us to 
obtain new families preserving those properties. This conceptual
setup allows us to answer general questions on behaviour of degrees 
in crossing-critical graphs: we show that, 
for any set of integers $D$ such that 
$\min(D)\geq 3$ and $3,4\in D$, and for any sufficiently large $k$, 
there exists a $k$-crossing-critical family such that the 
numbers in $D$ are precisely the vertex degrees that occur 
arbitrarily often in (large enough) graphs of this family.
Furthermore, even if both $D$ and some average degree in 
the interval $(3,6)$ are prescribed, $k$-crossing-critical
families exist for any sufficiently large $k$.
\end{abstract}

%%%%%%%%%%%%%%%%%%%%%%%%%%%%%%%%%%%%%%%%%%%%%%%%%%%%%%%%%%%%%%%%%%%%%%%%%%% 

%%%%%%%%%%%%%%%%%%%%%%%%%%%%%%%%%%%%%%%%%%%%%%%%%%%%%%%%%%%%%%%%%%%%%%%%%%% 

\section{Introduction} 

Reducing the number of crossings in a drawing of a graph is considered one
of the most important drawing aesthetics.
Consequently, a great deal of research work has been invested into
understanding what forces the number of edge crossings in a drawing of
the graph to be large.
There exist strong quantitative lower bounds, such as the famous Crossing
Lemma \cite{Ajtai19829,cit:leighton}.
However, the quantitative bounds typically show their strength only in
dense graphs, while in the area of graph drawing, we often deal with graphs
having few edges.

The reasons for sparse graphs to have many crossings in any drawing
are structural (there is a lot of ``nonplanarity'' in them).
These reasons can be understood via so called {\em$k$-crossing-critical}
graphs, which are the subgraph-minimal graphs that require at least~$k$ edge crossings
(the ``minimal obstructions'').
While there are only two $1$-crossing-critical graphs, up to
subdivisions---the Kuratowski graphs $K_5$ and $K_{3,3}$---it 
has been known already since \v{S}ir\'{a}\v{n}'s~\cite{cit:siran}
and Kochol's~\cite{cit:kochol} constructions that, for any $k\geq2$, the structure of $k$-crossing-critical
graphs is quite rich and non-trivial.

Although $2$-crossing-critical graphs can be efficiently (although not
easily) characterized~\cite{cit:2critchar}, 
a full description for any $k\geq3$
is clearly out of our current reach.
Consequently, research has focused on interesting properties shared by all
$k$-crossing-critical graphs (for certain $k$); 
successful attempts include, e.g.,
\cite{cit:embgrids,cit:nestedcycles,cit:pw,cit:starsbonds,cit:rtcrit}.
While we would like to establish as many specific properties of
crossing-critical graphs as possible, 
the reality unfortunately seems to be
against it. Many desired and conjectured properties of crossing-critical
graphs have already been disproved by often complex and sophisticated
constructions showing the odd behaviour of crossing-critical families,
e.g.~\cite{cit:dvorakmohar,cit:pathcrit,cit:newinfcrit,cit:avgcr4}.

We study properties of {\em infinite families} of 
$k$-crossing-critical graphs, for fixed
values of~$k$, since sporadic ``small'' examples 
of $k$-crossing-critical graphs tend to
behave very wildly for every $k>1$.
Among the most studied such properties are those related to vertex
degrees in the critical families, see
\cite{cit:avgcrit,cit:dvorakmohar,cit:nestedcycles,cit:newinfcrit,cit:avgcr4}.
Often the research focused on the average degree 
a $k$-crossing-critical family
may have---this rational number clearly falls into the interval $[3,6]$
if we forbid degree-$2$ vertices.
It is now known \cite{cit:nestedcycles} that the true values 
fall into the open interval $(3,6)$, and all the rational values in this
interval can be achieved~\cite{cit:avgcrit}.
However, for a fixed $k$, one cannot come 
arbitrarily close to 6 \cite{cit:nestedcycles}.

In connection with the proof of bounded pathwidth 
for $k$-crossing-critical families
\cite{cit:pathcrit,cit:pw}, it turned out to be a fundamental question whether
$k$-crossing-critical graphs have maximum degree bounded in~$k$.
The somehow unexpected negative answer was given by Dvo\v{r}\'ak and
Mohar~\cite{cit:dvorakmohar}.
In 2007, Bokal noted that all the known (by that time) constructions of
infinite $k$-crossing-critical families seem to use only vertices of degrees
$3,4,6$, and he asked what other degrees can occur frequently
(see the definition in Section~\ref{sc:gd:tools}) in
$k$-crossing-critical families.
Shortly after that Hlin\v en\'y extended his previous
construction~\cite{cit:pathcrit} to include an arbitrary combination
of any even degrees~\cite{cit:newinfcrit}, for sufficiently large~$k$. 
The characterization of 2-crossing-critical graphs \cite{cit:2critchar} implied 
that also vertices of degree $5$ occur arbitrarily often in $2$-crossing-critical graphs.

Though, \cite{cit:newinfcrit} answered only the easier half of Bokal's
question, and it remained a wide open problem of whether there exist
infinite $k$-crossing-critical families whose members contain many vertices
of odd degrees greater than~$5$.
Our joint investigation has recently led to an ultimate positive answer.

The contribution and new results of our paper can be summarized as follows:
\begin{itemize}
\item In Section~\ref{sc:gd:tools}, we review the tools which are commonly used
in constructions of crossing-critical families.
\item Section~\ref{sc:construction} presents the key new contribution---a
construction of crossing-critical graphs with repeated occurrence of any
prescribed odd vertex degree (Proposition~\ref{pro:degreesodd} and Theorem~\ref{thm:constrodd}).
\item In Section~\ref{sc:freq}, we combine the new construction of
Section~\ref{sc:construction} with previously known constructions to prove
the following:
for any set of integers $D$ such that $\min(D)=3$ and $3,4\in D$,
and for all sufficiently large $k$, there exists an infinite
$k$-crossing-critical family such that the numbers in $D$ are precisely the vertex
degrees which occur frequently in this family
(Theorem~\ref{thm:alluniversal}).
\item We extend the previous results in Section~\ref{sc:average}
to include an \mbox{exhaustive} discussion of possible average vertex degrees
attained by our degree-restricted crossing-critical families
(Theorem~\ref{thm:alluniversal-avgdeg}).
\item Then, in Sections~\ref{sc:2cc} and~\ref{sc:2ccdeg},
we pay special attention to infinite families of $2$-crossing-critical graphs 
and provide an exhaustive survey of their degree-related properties.
\item Finally, in concluding Section~\ref{sc:final}, we 
list some remaining interesting open questions.
\end{itemize}

%%%%%%%%%%%%%%%%%%%%%%%%%%%%%%%%%%%%%%%%%%%%%%%%%%%%%%%%%%%%%%%%%%%%%%%%%%% 
\section{Preliminaries}
\label{sc:gd:tools} 

We consider {\em finite multigraphs} without loops by default
(i.e., we allow multiple edges unless we explicitly call a graph {\em simple}),
and use the standard graph terminology otherwise.
The \DEF{degree} of a vertex~$v$ in a graph $G$ is the number of edges of $G$ incident
to~$v$ (cf.~multigraphs), and the \DEF{average degree} of $G$ is the average
of all the vertex degrees of~$G$.

\subsection{Crossing number}
In a {\em drawing} of a graph $G$, the vertices of $G$ are points 
and the edges are simple curves joining their endvertices.
It is required that no edge passes through a vertex, 
and no three edges cross in a common point.
The {\em crossing number} $\crn(G)$ of a graph $G$
is the minimum number of crossing points of edges
in a drawing of $G$ in the plane.
For $k\in\NN$, we say that a graph $G$ is {\em $k$-crossing-critical}, 
if $\crn(G)\geq k$ but, for every edge $e$ of $G$, $\crn(G-e)<k$.

Note that a vertex of degree $2$ in $G$ is not relevant for a drawing
of~$G$ and for the crossing number, and we will often replace such vertices
by edges between their two neighbours.
Since also vertices of degree $1$ are irrelevant for the crossing number,
it is quite common to assume minimum degree~$3$.

\subsection{Degree-universality}

The following terms formalize a vague notion that a certain vertex degree
occurs frequently or arbitrarily often in an infinite family.
For a finite set $D\subseteq \NN$, we say that a
family of graphs $\cF$ is \DEF{$D$-universal}, if and only if, for every
integer $m$, 
there exists a graph $G\in\cF$ such that, for every $d\in D$, $G$ has at least $m$ vertices of degree $d$.
It follows easily that $\cF$ has infinitely many such graphs.  

Clearly, if $\cF$ is $D$ universal and $D'\subseteq D$, then $\cF$ is also
$D'$-universal. The family of all sets $D$, for which a given $\cF$ is
$D$-universal, therefore forms a poset under relation $\subseteq$.
Maximal elements of this poset are of particular interest, and for 
``well-behaved'' $\cF$, these maximal elements are finite and unique. 
We distinguish this case with the following definition: 
$\cF$ is \DEF{$D$-max-universal}, if it is
$D$-universal, there are only finitely many degrees appearing in
graphs of $\cF$ that are not in $D$, and there exists 
an integer $M$, such that any degree not in $D$
appears at most $M$ times in any graph of $\cF$.

Note that if $\cF$ is both $D$-max-universal and $D'$-max-universal, then $D=D'$.
It can also be easily seen that if $\cF$ is $D$-max-universal,
then there exists infinite $\cF'\subseteq\cF$ such that, for any $m$,
{\em every} sufficiently large member of $\cF'$ has, for each $d\in D$, at least $m$ vertices of
degree $d$.
Though, we do not specifically mention this property in the formal
definition.

\subsection{Tools for constructing crossing-critical graphs} 
\label{sc:tools}

A principal tool used in construction of crossing-critical graphs are tiles. They 
were used already in the early papers on infinite families of crossing-critical 
graphs by Kochol \cite{cit:kochol} and Richter and Thomassen \cite{cit:rtcrit},
although they were formalized only in the work of Pinontoan and Richter 
\cite{cit:pltile,cit:gentile}, answering
Salazar's question \cite{cit:avgcr4} on average degrees in infinite families of 
$k$-crossing-critical graphs. Bokal built upon these results to fully settle
Salazar's question when combining tiles with zip product \cite{cit:avgcrit}. 
Also a recent result that all large 2-crossing-critical graphs are composed of 
large multi-sets of specific 42 tiles \cite{cit:2critchar} demonstrates that
tiles are intimately related to crossing-critical graphs. In this section, 
we summarize the known results from \cite{cit:avgcrit,cit:2critchar,cit:pltile}, 
which we need for our constructions.

Tiles are essentially graphs equipped with two sequences of vertices that are 
identified
among tiles or within a tile in order to, respectively, form new tiles or
tiled graphs. The tiles can be drawn in the unit square respecting the order of 
these sequences of vertices, thus providing special, restricted drawings of tiles.
Due to the restriction, the crossing number of these special drawings is an upper
bound to the crossing number of either underlying graphs, or the graphs obtained by
identifying these specific vertices. The formal concepts allowing these operations
are summarized in the following definition and the lemma immediately
after it:

\begin{definition}
Let $\lambda=(\lambda_1, \ldots, \lambda_l)$ and $\rho=(\rho_1, \ldots, \rho_r)$ 
be two sequences of distinct vertices of a graph $G$, where no vertex of $G$ 
appears in both $\lambda$ and $\rho$. 
\begin{enumerate}\parskip0pt
\item For any sequence $\lambda$, let $\bar{\lambda}$ denote its reversed sequence.
\item A \DEF{tile} is a triple $T=(G,\lambda, \rho)$.
\item The sequence of vertices $\lambda$ is called the \DEF{left wall} and the sequence of 
vertices $\rho$ is called the \DEF{right wall} of $T$.
\item A \DEF{tile drawing} of a tile $T=(G,\lambda, \rho)$ is a drawing of $G$ in unit 
square $[0,1]\times [0,1]$ such that:
\begin{itemize}
\item[--] all vertices of the left wall are drawn in $\left \{ 0\right \}\times [0,1]$ 
and all vertices of the right wall are drawn in $\left \{ 1\right \}\times [0,1]$;

\item[--] the left wall and the right wall have both decreasing $y$-coordinates.
\end{itemize}

\item The \DEF{tile crossing number $\text{cr}(T)$} of a tile $T$ is the smallest crossing 
number over all tile drawings of $T$.

\item A tile $T=(G, \lambda, \rho)$ is \DEF{compatible} with a tile $T'=(G', \lambda', \rho')$ 
if $|\rho|=|\lambda'|$ and \DEF{cyclically-compatible} if it is compatible with itself.

\item A sequence of tiles $(T_0, \ldots, T_m)$ is \DEF{compatible}, if, for $i=0, \ldots, m-1$, tiles $T_i$ and $T_{i+1}$ 
are compatible. 
It is \DEF{cyclically-compatible} if also $T_m$ is compatible with $T_0$.

\item The \DEF{join} of two compatible tiles $T=(G,\lambda, \rho)$ and $T'=(G', \lambda', \rho')$ 
is defined as the tile $T \otimes T'=(G \otimes G', \lambda, \rho')$, where 
$G\otimes G'$ represents the graph obtained from the union of graphs $G$ and $G'$,
 by identifying, for $i=1,\ldots, |\rho|$, $\rho_i$ with $\lambda'_i$. 
 If a vertex of degree 2 is introduced, then all maximal paths whose internal vertices are all of degree 2 are contracted to a single edge.
 Introduced double edges are retained.

\item Since the operator $\otimes$ is associative, the \DEF{join} $\otimes \cT$ of a compatible 
sequence of tiles $\cT=(T_0, \ldots, T_m)$ is defined as $\otimes \cT=T_0 \otimes \ldots \otimes T_m$.

\item Let $T=(G, \lambda, \rho)$ be a cyclically-compatible tile. The \DEF{cyclization} $\circ T$
 of a tile $T$ is the graph $G$ obtained by identifying, for $i=1, \ldots, |\lambda|$, $\lambda_i$ with $\rho_i$.

\item Let $\cT=(T_0, \ldots, T_m)$ 
be a cyclically-compatible sequence of tiles with
$T_0=(G,\lambda, \rho)$, $T_m=(G',\lambda', \rho')$. 
The \DEF{cyclization} of $\cT$ is defined as 
$\circ \cT=\circ{(T_0\otimes \ldots \otimes T_m)}$.

\item Let $T=(G, \lambda, \rho)$ be a tile. The \DEF{right-inverted} tile 
$T^{\updownarrow}$ is the tile $(G, \lambda, \bar{\rho})$ and the \DEF{left-inverted} 
tile $^{\updownarrow} T$ is the tile $(G, \bar{\lambda}, \rho)$. 
The \DEF{inverted} tile of $T$ is the tile $^{\updownarrow}T^{\updownarrow}=(G, \bar{\lambda}, \bar{\rho})$ 
and the \DEF{reversed} tile of $T$ is the tile $T^{\leftrightarrow}=(G,\rho, \lambda)$.
\item For a compatible sequence of tiles $\cT$, the \DEF{twist} is 
$\cT^{\updownarrow}=(T_0, \ldots, T_{m}^{\updownarrow})$, and the \DEF{$i$-cut} 
of $\cT$ is $\cT/i=(T_{i+1},\ldots ,T_m, T_0,\ldots,T_{i-1})$.

\end{enumerate}
\end{definition}

\begin{lemma}[\cite{cit:pltile}]
\label{le: lema1}
Let $T$ be a cyclically-compatible tile. Then, $\crn(\circ T)\leq \tcrn(T)$. 
Let $\cT=(T_0,\ldots, T_m)$ be a compatible sequence of tiles. Then, 
$\tcrn(\otimes \cT)\leq \sum_{i=0}^{m}\tcrn(T_i)$.
\end{lemma}

The above Lemma applies without any information on the internal structure of the tiles.
However, by exploiting their internal structure (planarity and enough connectivity),
we can also prove a lower bound on the tile crossing number, which can, with sufficiently
many tiles, be exploited for the lower bound on the crossing number of the graph resulting 
from the tile. Prerequisites for these applications are summarized in the following
definition and applied in the theorem that follows.
 
\begin{definition}
Let $T=(G, \lambda, \rho)$ be a tile. Then:
\begin{enumerate}
\item $T$ is \DEF{connected} if $G$ is connected. 
\item $T$ is \DEF{planar} if $\tcrn(T)=0$. 
\item $T$ is \DEF{perfect} if the following holds:
\begin{itemize}
\def\labelitemi{--}
\item $|\lambda|=|\rho|$;
\item $G-\lambda$ and $G-\rho$ are connected;
\item for every $v\in\lambda$ there is a path to the right wall $\rho$ 
in $G$ internally disjoint from $\lambda$ and for every $u\in\rho $ there 
is a path to the left wall $\lambda$ in $G$ internally disjoint from $\rho$;
\item for every $0\leq i<j\leq|\lambda|$, there is a pair of disjoint paths, 
one joining $\lambda_i$ and $\rho_i$, and the other joining $\lambda_j$ and $\rho_j$.
\end{itemize}
\end{enumerate}
\end{definition}

\begin{theorem}[\cite{cit:avgcrit}]
\label{th: theorem1}
Let $\cT=(T_0, \ldots,T_{\ell}, \ldots, T_m)$ be a cyclically-compatible sequence of 
tiles. Assume that, for some integer $k\geq 0$, the following hold: $m\geq 4k-2$ 
 and, for every $i\in \left\{0,\ldots, m\right\} \backslash \left\{\ell\right\}$,
 $\tcrn(\otimes (\cT / i)) \geq k$, and the tile $T_i$ is a perfect planar tile. Then, $\crn(\circ\cT)\geq k$. 
\end{theorem}

This theorem can yield exact crossing number under the assumptions of the next corollary.

\begin{corollary}[\cite{cit:avgcrit}]
\label{co: planarity}
Let $\cT=(T_0, \ldots,T_{\ell}, \ldots,  T_m)$ be a cyclically-compatible sequence 
of tiles and let $k=\min_{i\in \left\{0,\ldots, m\right\} \backslash \left\{\ell\right\}}\tcrn(\otimes(\cT / i))$. 
If $m\geq 4k-2$ and, for every $i\in \left\{0,\ldots, m\right\} \backslash \left\{\ell\right\}$,
the tile $T_i$ is a perfect planar tile, 
then $\crn(\circ\cT)=k$.
\end{corollary}

Exact lower bounds facilitate establishing criticality of the tiles and graphs,
as the smallest drop in crossing number suffices for criticality of an edge.
For combinatorially handling the criticality of the
constructed graph on the basis of the properties of tiles, 
we introduce \DEF{degeneracy}
of tiles and criticality of sequences of tiles as follows:

%K-DEGENERATE 
\noindent
\begin{definition}\label{def:criticalSeq}
\begin{enumerate}
\item A tile $T$ is \DEF{$k$-degenerate} if it is perfect, planar and, for any $e\in E(T)$, $\tcrn(T^{\updownarrow}-e)<k$.
\item A sequence $\cT=(T_0, \ldots, T_m)$ is \DEF{$k$-critical} if, for every $i=0, \ldots, m$, the tile $T_i$ is 
$k$-degenerate and 
$\min_{i\in \left\{0,\ldots, m-1\right\}}\tcrn(\otimes((\cT/i)^{\updownarrow}))\geq k$.
\end{enumerate}
\end{definition}
\noindent
Using these concepts, Corollary \ref{co: planarity} can be applied to establish 
criticality of graphs resulting from crossing critical sequences of tiles 
or from degenerate tiles.
\begin{corollary}[\cite{cit:avgcrit}]
\label{co: criticalGraph}
Let $\cT=(T_0, \ldots, T_m)$ be a $k$-critical sequence of tiles. Then, 
$T=\otimes \cT$ is a $k$-degenerate tile. If $m\geq 4k-2$ and $\cT$ is 
cyclically-compatible, then $\circ{(T^{\updownarrow})}$ is a $k$-crossing-critical 
graph.
\end{corollary}

\noindent
To estimate the tile crossing number, we use an informal tool called \textit{gadget}. 
This can be any structure inside of a tile $T$, which guarantees a certain number of 
crossings in every tile drawing of $T$. The gadgets we use are twisted pairs of paths,
guaranteeing one crossing each, and staircase strips of width $n$, guaranteeing 
$\binom{n}{2}-1$ crossings.

\noindent
\begin{definition}
A \DEF{traversing path} in a tile $T=(G,\lambda,\rho)$ is a path $P$ in the graph $G$, for which
there exist indices $i(P)\in \left\{ 1, \ldots, |\lambda| \right\}$, and 
$j(P)\in \left\{ 1, \ldots, |\rho|\right\}$, so that $P$ is a path from 
$\lambda(P)=\lambda_{i(P)}$ to $\rho(P)=\rho_{j(P)}$ and $\lambda(P)$ and 
$\rho(P)$ are the only wall vertices of $P$.

A pair of disjoint traversing paths $\left\{P,Q \right\}$ is \DEF{twisted} if $i(P)<i(Q)$ 
and $j(P)>j(Q)$, and \DEF{aligned} otherwise. A family $\cW$ of pairs of disjoint traversing
paths is \DEF{aligned}, if all the pairs in $\cW$ are aligned. The family is \DEF{twisted}, 
if all the pairs are twisted.
\end{definition}
\noindent

The disjointness of a twisted pair $\left\{P,Q \right\}$ implies one crossing in any tile drawing of $T$. 
This is generalized to twisted families in the following lemma: 
\noindent
\begin{lemma}[\cite{cit:avgcrit}]
\label{le: twisted}
Let $\cW$ be a twisted family in a tile $T$, such that no edge occurs
in two distinct paths of $\>\cup\cW$. Then, $\tcrn(T)\geq |\cW|$.
\end{lemma}

The following definition presents a staircase tile, adapted from \cite{cit:avgcrit}. Such detailed definition is needed as this tile is later used as a part of our new constructed tile defined in Section \ref{sc:construction}.
A reader should understand it quickly when referring to Figure~\ref{fig:staircaseTile}.

\begin{figure}[!ht]
\begin{center}
\includegraphics[width=0.7\hsize]{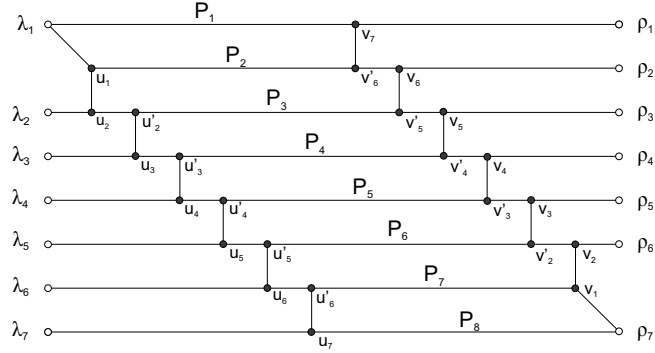}
\end{center}
\caption{A staircase tile $S_{8}$. The wall vertices are drawn in white and internal vertices in black.}
\label{fig:staircaseTile}
\end{figure}
\vspace*{-3mm}%
%STAIRCASE TILE
\begin{definition}[\cite{cit:avgcrit}]
\label{df: staircaseTile}
Let $n\in \NN$, $n\geq 3$. The \DEF{staircase tile} of width $n$ is a tile $S_{n}=(G, \lambda, \rho)$ with 
$\lambda = (\lambda_1, \lambda_2, \ldots, \lambda_{n-1})$, $\rho = (\rho_1, \rho_2, \ldots, \rho_{n-1})$ for which the following holds:

	\begin{itemize}
	\item $S_{n}$ consists of a sequence of traversing paths $\cP=\left\{P_1, P_2, \ldots, P_{n}\right\}$ with the property:
	\begin{itemize}
	\item $\lambda(P_1)=\lambda_{1}$, $\rho(P_1)=\rho_{1}$ and $\lambda(P_n)=\lambda_{n-1}$, $\rho(P_n)=\rho_{n-1}$,
    \item for $i=2, \ldots, n-1$, $\lambda(P_i)=\lambda_{i-1}$ and $\rho(P_i)=\rho_{i}$.
	 
\end{itemize}
	
\item The only non-wall vertices of $S_n$ are, for $i=1,n-1$, $u_i\in P_{i+1}$, $v_{i}\in P_{n-i}$, and, for $i = 2, \ldots, n-2$, $u_i, u'_i \in P_{i+1}$ and $v_i,v'_i\in P_{n-i}$.

\item For $i = 2, 3, \ldots, n-2$, the position of a non-wall vertices of $S_{n}$ is such that $\left|e(u_i)\cap e(\lambda_{i}) \right|=1$, $\left|e(u_i)\cap e(u'_i) \right|=1$, $\left|e(v_i)\cap e(\rho_{n-i}) \right|=1$ and $\left|e(v_i)\cap e(v'_i) \right|=1$.

\item The additional edges are $u_1u_2$, $v_1v_2$ and, for $i=2, \ldots, n-2$, $u'_{i}u_{i+1}$ and $v_{i}v'_{i+1}$.
\end{itemize}
\end{definition}
For $n\geq 3$, a staircase tile $S_n$ is a perfect planar tile. 
\begin{definition}[\cite{cit:avgcrit}]
\label{df: staircaseSequence}
Let $n\geq 3$ be integer and let $m\geq 3$ be an odd integer. The \DEF{staircase sequence} of length $m$ is defined as $\cS_{n,m}=(S_n, {}^{\updownarrow}S_n{}^{\updownarrow}, S_n, {}^{\updownarrow}S_n{}^{\updownarrow}, \ldots, {}^{\updownarrow}S_n{}^{\updownarrow}, S_n)$ and the \DEF{staircase strip graph} as $S(n,m)=\circ (\cS_{n,m}{}^{\updownarrow})$.
\end{definition}

\begin{proposition}[\cite{cit:avgcrit}]
\label{pr: staircase}
The staircase strip graph, $S(n,m)$, of width n and odd length $m\geq 4\binom{n}{2}-5$ is a crossing-critical graph with $\crn(S(n,m))=\binom{n}{2}-1$.
\end{proposition}

This concludes our discussion of known results on tiles in graphs. 
Tiled graphs are joined
together using zip product construction \cite{cit:zipprod,cit:avgcrit}. 
We use the version restricted to vertices of degree three, as introduced 
in \cite{cit:newinfcrit}.
\begin{definition}
\label{df:zip}
For $i\in \left \{ 1,2 \right\}$, let $G_i$ be a graph and let 
$v_i \in V(G_i)$ be its vertex of degree 3, such that $G_i-v_i$ is connected and $v_i$ is incident only to simple edges. 
We denote the neighbours of $v_i$ by $u_j^i$ for $j\in  \left \{ 1,2,3 \right\}$. 
The zip product of $G_1$ and $G_2$ according to vertices $v_1$, $v_2$ and their 
neighbours, is obtained from the disjoint union of $G_1-v_1$ and $G_2-v_2$ by 
adding three edges $u_1^1u_1^2$, $u_2^1u_2^2$, $u_3^1u_3^2$. 
\end{definition}

While crossing number is super-additive over general zip products only under a 
technical connectivity condition, the following theorem holds for zip products 
of degree (at most) three:
\begin{theorem}[\cite{cit:zip3}]\label{thm:zip3}
Let $G$ be a zip product of $G_1$ and $G_2$ as in Definition \ref{df:zip}. 
Then, $\crn(G)=\crn(G_1)+\crn(G_2)$. 
Consequently, if, for $i=1,2$, $G_i$ is $k_i$-crossing-critical,
then $G$ is $(k_1+k_2)$-crossing-critical.
\end{theorem}

%%%%%%%%%%%%%%%%%%%%%%%%%%%%%%%%%%%%%%%%%%%%%%%%%%%%%%%%%%%%%%%%%%%%%%%%%%%

\section{Crossing-Critical Families with High Odd Degrees}
\label{sc:construction} 

We first present a new construction of a
crossing-critical family containing many vertices of an arbitrarily
prescribed odd degree (recall that the question of an existence of such
families has been the main motivation for this research).

\begin{figure}[!ht]
\centering
\vspace*{3mm}%
\includegraphics[width=0.5\hsize]{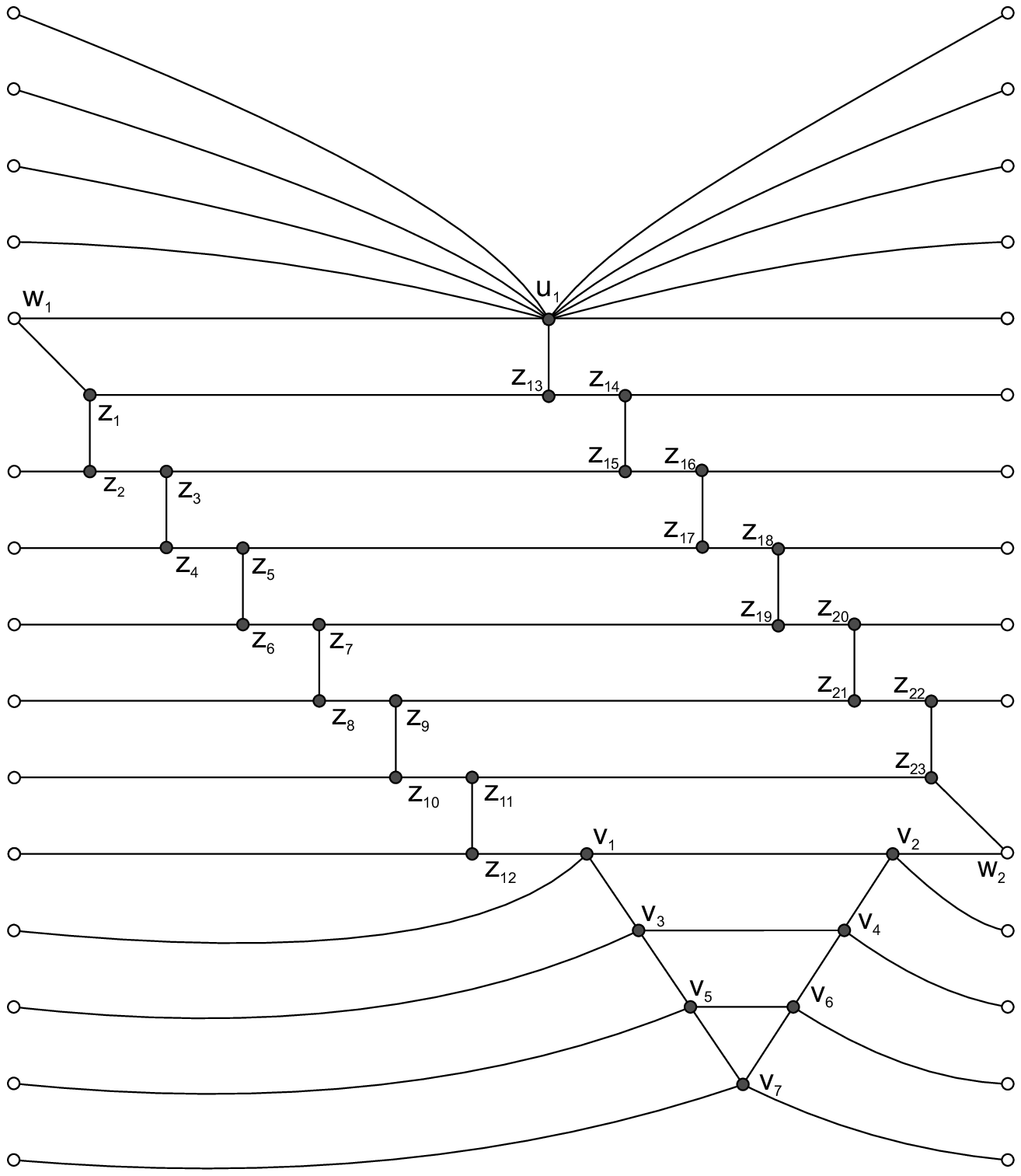}
\hbox{\vbox to 72mm{\hsize=5mm
$P'_4$\par\vspace{1mm} $P'_3$\par\vspace{1mm}$P'_2$\par\vspace{1mm}$P'_1$\par\vspace{0.5mm}
$S'_1$\par\vspace{0.8mm}$S'_2$\par\vspace{0.8mm}
$S'_3$\par\vspace{0.8mm}$S'_4$\par\vspace{0.8mm}
$S'_5$\par\vspace{0.8mm}$S'_6$\par\vspace{0.8mm}
$S'_7$\par\vspace{0.8mm}$S'_8$\par\vspace{1mm}
$Q'_1$\par\vspace{1mm} $Q'_2$\par\vspace{1mm}$Q'_3$\par\vspace{1mm}$Q'_4$
}}%
\caption{A tile drawing of $H_{4,8}$. The wall vertices are drawn in white and internal vertices in black.}
\label{fig:tileHln}
\end{figure}

The construction defines a graph $G(\ell,n,m)$ with three integer parameters
$\ell\geq 1$, $n \geq 3$ and odd $m\geq3$, as follows. 
There is a tile $H_{\ell,n}$, with the walls of size $2\ell+n-1$,
which is illustrated in Figure~\ref{fig:tileHln}. 
Formally, $H_{\ell,n}$ consists of $2\ell+n$ pairwise edge disjoint paths,
grouped into three families $P_1',\ldots,P_\ell'$,
$Q_1',\ldots,Q_\ell'$, and $S_1',\ldots,S_n'$,
and an additional set $F'$ of $2(n-2)$ edges not on these paths, see Figure~\ref{fig:tileHln}. We continue with more detailed definition:
\begin{itemize}
\item
The paths $S_1',\ldots,S_n'$ are pairwise vertex-disjoint except that
$S_1'$ shares one vertex with $S_2'$ ($w_1$ in Figure~\ref{fig:tileHln}) and $S_{n-1}'$ shares one vertex with $S_n'$ ($w_2$ in Figure~\ref{fig:tileHln}).
The additional $2$ edges of $F'$ appear between vertices of paths: $S_{2}'$ and $S_3'$ (edge $z_1z_2$ in Figure~\ref{fig:tileHln}), $S_{1}'$ and $S_2'$ (edge $u_1z_{2(n-1)-1}$ in Figure~\ref{fig:tileHln}). If $n>3$, then, for $i=2, \ldots, n-2$, additional $n-3$ edges of $F'$ appear between vertices of paths $S_{i+1}'$ and $S_{i+2}'$ (edges $z_{2i-1}z_{2i}$ in Figure~\ref{fig:tileHln}) and additional $n-3$ edges of $F'$ appear between vertices of paths $S_{i}'$ and $S_{i+1}'$ (edges $z_{2(n-3)+2i}z_{2(n-3)+2i+1}$ in Figure~\ref{fig:tileHln}).
\item
The union $S_1'\cup\ldots\cup S_n'\cup F'$ is (consequently) a subdivision
of the aforementioned staircase tile from Definition~\ref{df: staircaseTile}.
\item
Let $\lambda=(\lambda_{1}, \ldots, \lambda_{2\ell+n-1})$ be the left wall and $\rho=(\rho_{1}, \ldots, \rho_{2\ell+n-1})$ the right wall of $H_{\ell,n}$. The paths $P_1',\ldots, P_\ell'$ are ordered such that, for $i=1, \ldots, n$, $\lambda(P_i)=\lambda_{\ell+1-i}$ and $\rho(P_i)=\rho_{\ell+1-i}$. The paths $Q_1',\ldots, Q_\ell'$ are ordered such that, for $i=1, \ldots, n$, $\lambda(Q_i)=\lambda_{\ell+n-1+i}$ and $\rho(P_i)=\rho_{\ell+n-1+i}$.
\item 
The paths $P_1',\ldots,P_\ell'$ all share the top-most vertex $u_1$ of
$S_1'$.
\item
The paths $Q_1'$ and $S_n'$ shares exactly two vertices of degree~$4$ ($v_1$ and $v_2$ in Figure~\ref{fig:tileHln}).
For $i=1,\ldots,\ell-2$, $Q_i'$ shares exactly two vertices
with $Q_{i+1}'$ and these shared
vertices are of degree~$4$,
as depicted in
Figure~\ref{fig:tileHln} (vertices $v_{i+2},\ldots,v_{2\ell-2}$). The paths $Q_{\ell-1}'$ and $Q_\ell'$ shares exactly one vertex of degree~$4$ ($v_{2\ell-1}$ in Figure~\ref{fig:tileHln}). 
\end{itemize}

\begin{figure}[!t]
\centering
\includegraphics[width=0.93\hsize]{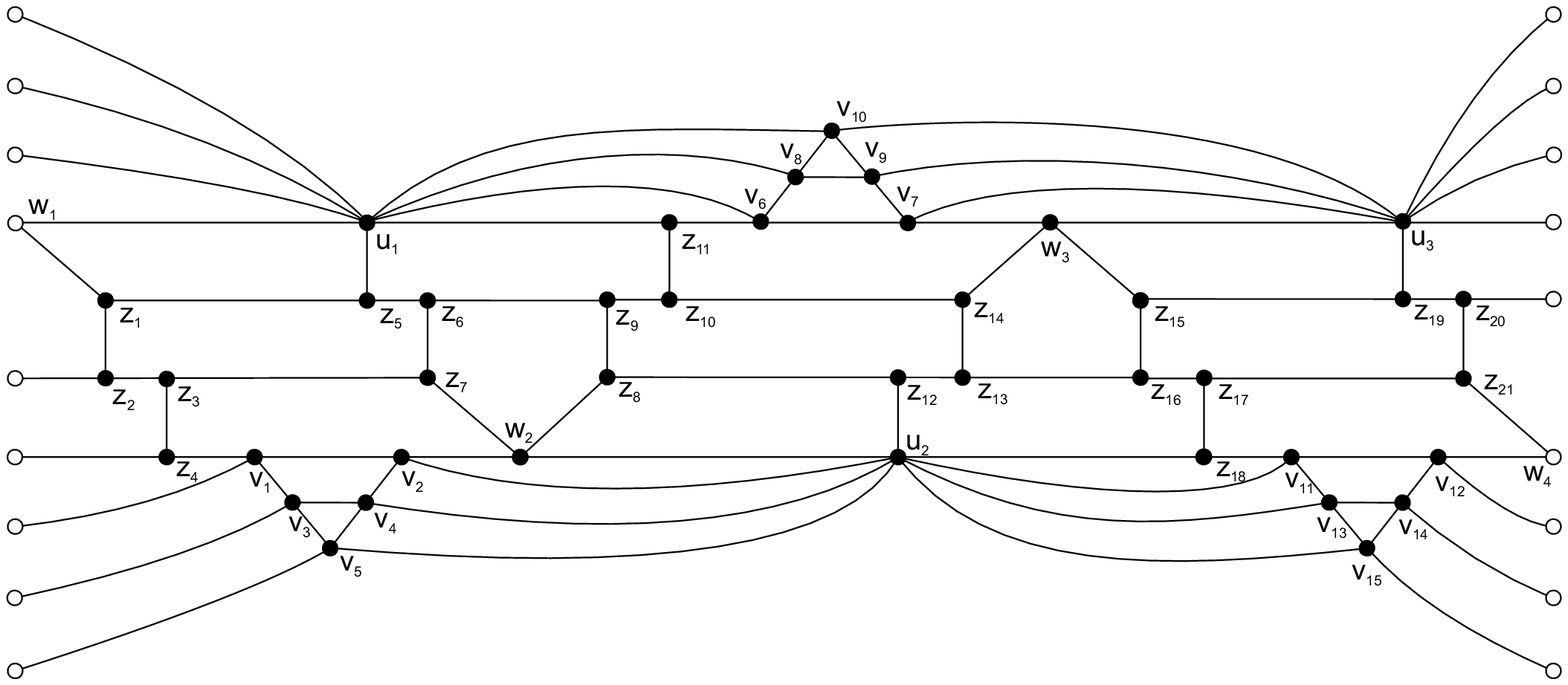}
\hbox{\vbox to 50mm{\hsize=5mm
$P''_3$\par\vspace{1.3mm}$P''_2$\par\vspace{1.3mm}$P''_1$\par\vspace{1.3mm}
$S''_1$\par\vspace{1.3mm}$S''_2$\par\vspace{2.3mm}
$S''_3$\par\vspace{1.8mm}$S''_4$\par\vspace{1.3mm}
$Q''_1$\par\vspace{1.3mm}$Q''_2$\par\vspace{1.3mm}$Q''_3$}}%
\caption{A tile drawing of $G_{3,4}$. The wall vertices are drawn in white and internal vertices in black.}
\label{fig:tileGln}
\end{figure}

Let $G_{\ell,n}$ be a tile composed of three copies of $H_{\ell,n}$ such that $G_{\ell,n}=H_{\ell,n}\otimes
	{}^{\updownarrow}H_{\ell,n}{}^{\updownarrow}\otimes H_{\ell,n}$, which is illustrated in Figure~\ref{fig:tileGln}.
Let $\cG(\ell,n,m)=(G_{\ell,n}, {}^{\updownarrow}G_{\ell,n}{}^{\updownarrow}, G_{\ell,n}
\ldots ,{}^{\updownarrow}G_{\ell,n}{}^{\updownarrow}, G_{\ell,n})$ be a sequence of
such tiles of length $m$, and let
$G(\ell,n,m)$ be constructed as the cyclization
$\circ\big(\cG(\ell,n,m)\,^{\updownarrow}\big)$.

In the degenerate case of $\ell=0$, the graph $G(0,n,m)$ is a staircase strip graph (see Definition~\ref{df: staircaseSequence}),
and $G(0,n,m)$ will be contained in $G(\ell,n,m)$ as a subdivision for every~$\ell$.

For $i=1,\ldots, \ell$ and $j=1,\ldots, n$, let $P_i'',Q_i'',S_j''$ denote the paths obtained by three copies of each of $P_i',Q_i',S_i'$ from $H_{\ell,n}$ as $P_i''=\otimes (P'_i, Q'_i, P'_i)$, $Q_i''=\otimes (Q'_i, P'_i, Q'_i)$ and $S_j''=\otimes (S'_j, S'_{n+1-j}, S'_j)$.
Then $P_1'',\ldots,P_\ell''$, $Q_1'',\ldots,Q_\ell''$, and
$S_1'',\ldots,S_n''$ are all traversing paths of $G_{\ell,n}$.\\
For $i=1, \ldots, \ell$ and $j=1,\ldots, n$, let $\bar{P_i},\bar{Q}_i,\bar{S}_i$ denote the paths which are obtained by $m$ copies of each of $P_i'',Q_i'',S_i''$ from $G_{\ell,n}$ as $\bar{P_i}=\otimes (P''_i, Q''_i, P''_i, \ldots, Q''_i, P''_i)$, $\bar{Q_i}=\otimes (Q''_i, P''_i, Q''_i, \ldots, P''_i, Q''_i)$ and $\bar{S_j}=\otimes (S''_j, S''_{n+1-j}, S''_j, \ldots, S''_{n+1-j}, S''_j )$.
Then $\bar{P_i},\bar{Q}_i,\bar{S}_i$ are all traversing paths of
$\otimes\cG(\ell,n,m)$.

The proof of the following basic properties is straightforward,
as attentive reader could easily verify from the illustrating pictures of $H_{\ell,n}$
(recall that degree-$2$ vertices are contracted in a tile join).

\begin{proposition}\label{pro:degreesodd}
For every $\ell\geq1$ and $n \geq 3$,
the tiles $H_{\ell,n}$, and hence also $G_{\ell,n}$, are perfect planar tiles.
The graph $G(\ell,n,m)$ has $3m(2\ell+4n-8)$ vertices,
out of which $3m\cdot2\ell$ have degree~$4$,~
$3m(4n-9)$ have degree~$3$, and remaining $3m$ vertices have degree~$2\ell+3$.
The average degree of $G(\ell,n,m)$ is 
$$\frac{5l+6n-12}{l+2n-4}\,. \eqno{\qed}$$
\end{proposition}

We conclude with the main desired property of the graph $G(\ell,n,m)$.
\begin{theorem}\label{thm:constrodd}
Let $\ell\geq 1$, $n \geq 3$ be integers.
Let $k=(\ell^{2}+\binom{n}{2}-1+2\ell(n-1))$ and let $m\geq 4k-1$ be odd.
Then the graph $G(\ell,n,m)$ is $k$-crossing-critical.
\end{theorem}

\begin{proof}
By using Theorem~\ref{th: theorem1} and symmetry, 
it suffices to prove the following:
\begin{enumerate}[I)]
\item $\tcrn\big(\!\otimes\cG(\ell,n,m)\,^{\updownarrow}\big)\geq k$, and
\item every edge of $G_{\ell,n}$ corresponding to one copy of
	$H_{\ell,n}$ in it is {\em critical}, meaning that,
	for every edge $e\in E(H_{\ell,n})\subseteq E(G_{\ell,n})$,
	$\tcrn(G_{\ell,n}{}^{\updownarrow}-e) <k$.
\end{enumerate}

Recall the pairwise edge-disjoint traversing paths $\bar{P_1},\ldots,\bar{P_\ell}$, 
$\bar{Q_1},\ldots,\bar{Q_\ell}$, and $\bar{S_1},\ldots,\bar{S_n}$ of the join
$\otimes\cG(\ell,n,m)$.
We define the following disjoint sets of pairs of these paths, 
such that each pair is formed by {\em vertex-disjoint} paths:
\begin{itemize}
\item $\cA=\left\{ \{\bar{P_i}, \bar{Q_j} \}: 1\leq i,j \leq \ell \right\}$
where $|\cA|=\ell^2$,
\item $\cB=\left\{ \{\bar{P_i}, \bar{S_j} \}: 1\leq i\leq \ell, 1<j\leq n \right\}$
where $|\cB|=\ell(n-1)$,
\item $\cC=\left\{ \{\bar{Q_i}, \bar{S_j} \}: 1\leq i\leq \ell, 1\leq j< n \right\}$
where $|\cC|=\ell(n-1)$.
\end{itemize}

Each pair in $\cA\cup\cB\cup\cC$ is twisted in 
$\otimes\cG(\ell,n,m)^{\updownarrow}$, and so these pairs account for at least
$|\cA|+|\cB|+|\cC|=2\ell(n-1)+\ell^2$ crossings in a tile drawing of
$\otimes\cG(\ell,n,m)^{\updownarrow}$, by Lemma~\ref{le: twisted}.
Importantly, each of these crossings involves at least one edge of
$R=\bar{P_1}\cup\ldots\cup \bar{P_\ell}\cup \bar{Q_1}\cup\ldots\cup \bar{Q_\ell}$.
The subgraph $\otimes\cG(\ell,n,m)-E(R)$ contains a subdivision of the staircase
strip $\otimes\cG(0,n,m)$.

\begin{figure}[!ht]
\begin{center}
\vspace*{-3mm}%
\includegraphics[width=0.45\hsize]{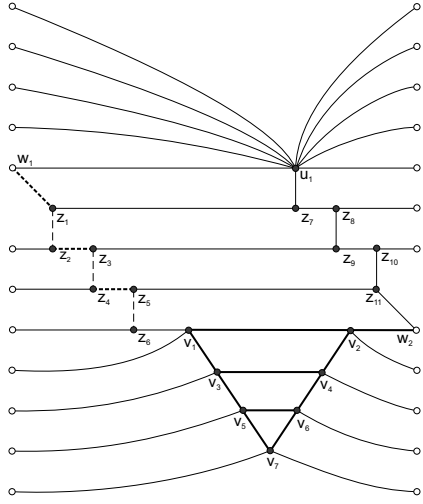}
\end{center}
\caption{A tile drawing of $H_{4,5}$ depicted with $4$ types of edges.}
\label{fig:CCEdges}
\end{figure}

Hence any tile drawing of $\otimes\cG(\ell,n,m)^{\updownarrow}$
contains at least another $\tcrn\big(\!\otimes\cG(0,n,m)^{\updownarrow}\big)$ 
crossings not involving any edges of $R$.
Since $\tcrn\big(\!\otimes\cG(0,n,m)^{\updownarrow}\big)=\binom{n}{2}-1$
by Proposition \ref{pr: staircase}, we get
$\tcrn\big(\!\otimes\cG(\ell,n,m)^{\updownarrow}\big)\geq
	\binom{n}{2}-1 + 2\ell(n-1)+\ell^2=k$,
thus proving (I).

\smallskip
To prove (II), we investigate the tile drawing in
Figure~\ref{fig:CriticalNew1}.
It is routine to count that a natural generalization of this drawing has
precisely $\binom{n-2}{2}+2(n-2)+2\ell(n-1)+\ell^2=k$ crossings,
and so it is optimal. Three types of crossings are presented in Figure~\ref{fig:CriticalNew1}, Figure~\ref{fig:CriticalNew2}, Figure~\ref{fig:CriticalNew3} and Figure~\ref{fig:CriticalNew4}:

\begin{itemize}
	\item grey triangles are the $\ell^2$ crossings of each pair in $\cA$;
	\item grey $4$-stars are the $2\ell(n-1)$ crossings of each pair in $\cB \cup \cC$;
	\item grey squares are the $\binom{n-2}{2}+2(n-2)=\binom{n}{2}-1$ crossings of edges in a staircase part, $G(0,n,m)$, of a graph $G(\ell,n,m)$. The $\binom{n-2}{2}$ crossings appear in the middle of Figure~\ref{fig:CriticalNew1} caused by edges, for $i=1, \ldots, n-2$, $z_{2i-1}z_{2n+2i-5}$. The $n-2$ crossings are caused by edge $w_1u_1$ and $n-2$ crossings by edge $z_{2n-4}w_2$.
\end{itemize}

\begin{figure}[!ht]
\begin{center}
\vspace*{3mm}%
\includegraphics[width=0.78\hsize]{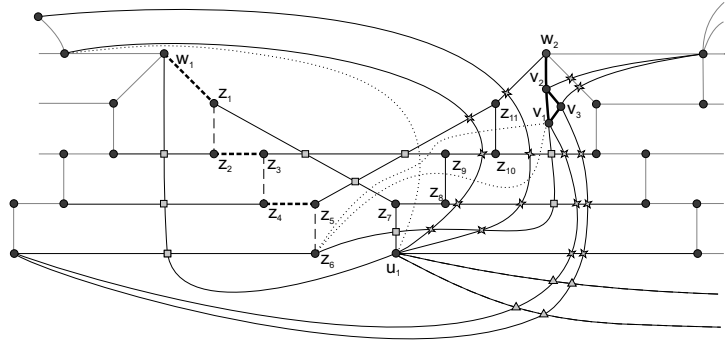}
\end{center}
\caption{A fragment of an optimal tile drawing of $G_{2,5}{}^{\updownarrow}$. Grey triangles, $4$-stars and squares present three types of crossings. For clearer presentation we use different $\ell$ compared to Figure~\ref{fig:CCEdges}. Here $\ell=2$ and $n=5$, which means that an optimal drawing has $29$ crossings: $\ell^2=4$ of triangles, $2\ell(n-1)=16$ of $4$-stars and $\binom{n}{2}-1=9$ of squares. Dotted lines show other possible renderings for certain edges. Note that we only change a single edge in each alternative drawing. The same observation holds for the following two figures.}
\label{fig:CriticalNew1}
\end{figure}
\begin{figure}[!ht]
\begin{center}
\vspace*{5mm}%
\includegraphics[width=0.78\hsize]{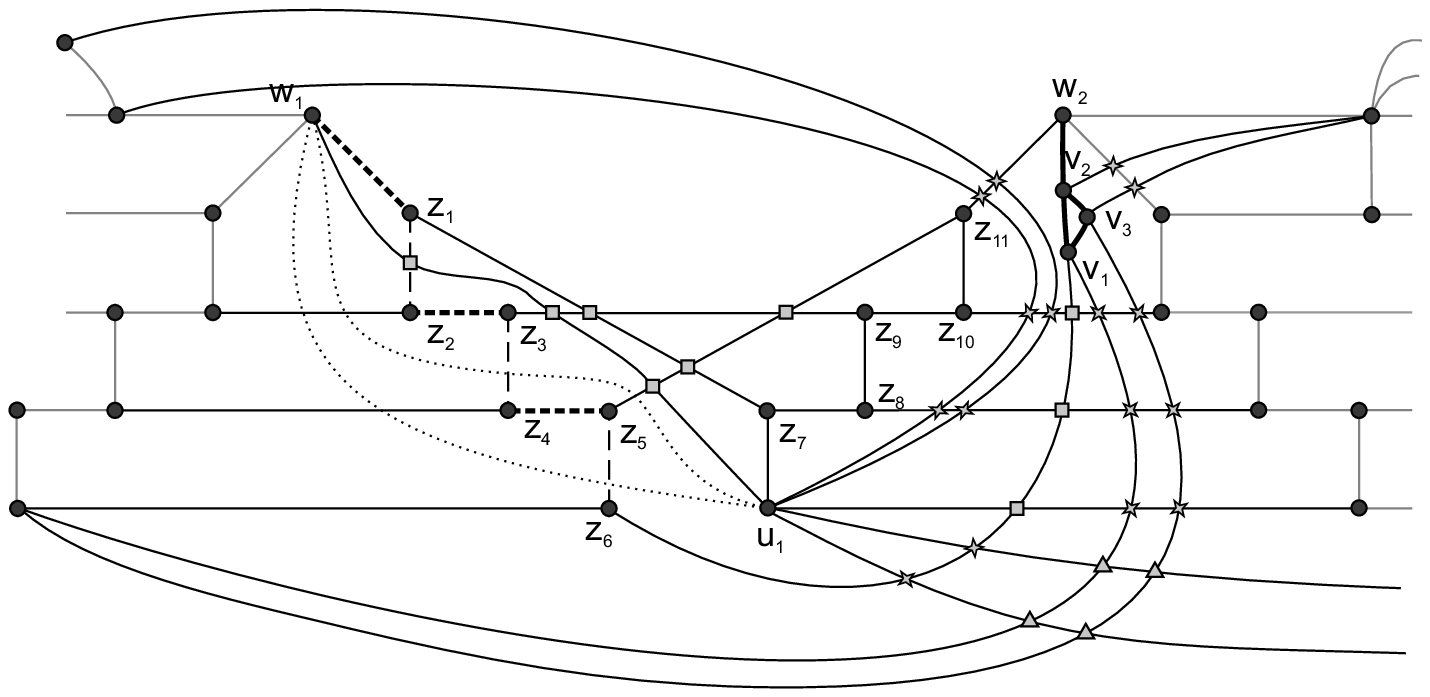}
\end{center}
\caption{A fragment of an optimal tile drawing of $G_{2,5}{}^{\updownarrow}$. See Figure~\ref{fig:CriticalNew1} for additional explanation.}
\label{fig:CriticalNew2}
\end{figure}
\begin{figure}[!ht]
\begin{center}
\includegraphics[width=0.8\hsize]{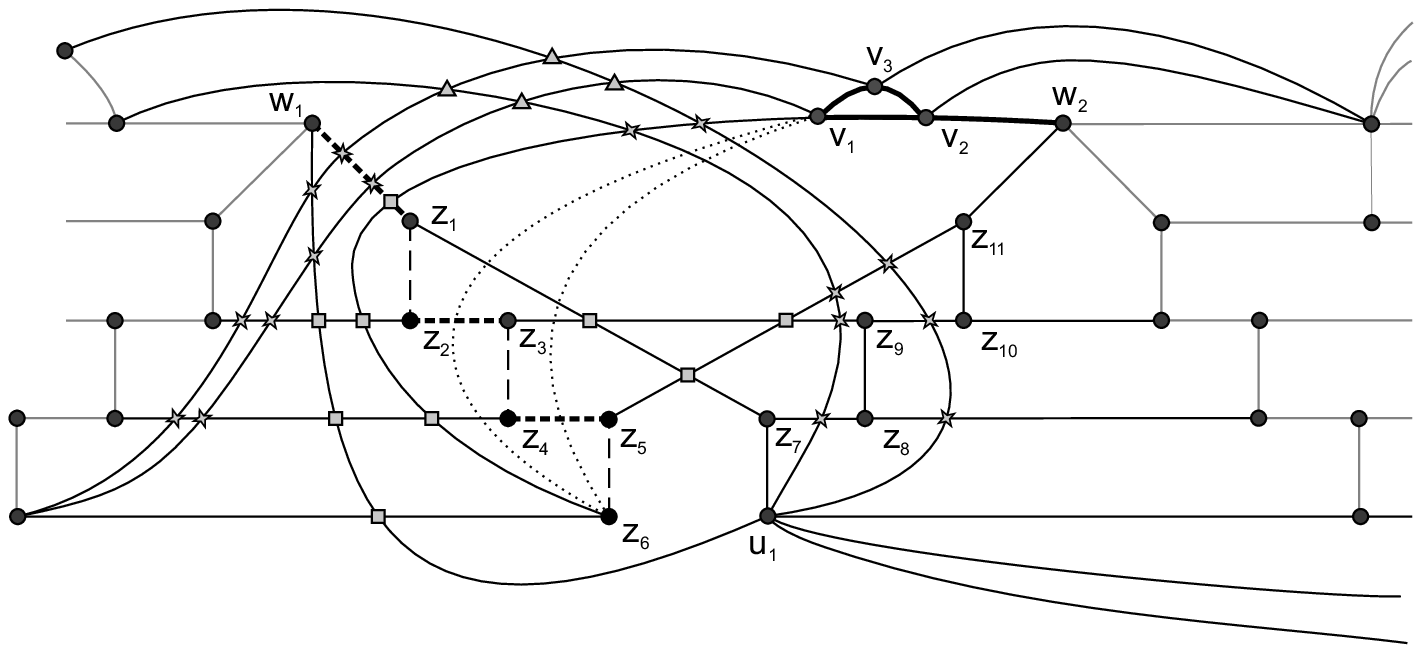}
\end{center}
\caption{A fragment of an optimal tile drawing of $G_{2,5}{}^{\updownarrow}$. See Figure~\ref{fig:CriticalNew1} for additional explanation.}
\label{fig:CriticalNew3}
\end{figure}
\begin{figure}[!ht]
\begin{center}
\vspace*{6mm}%
\includegraphics[width=0.8\hsize]{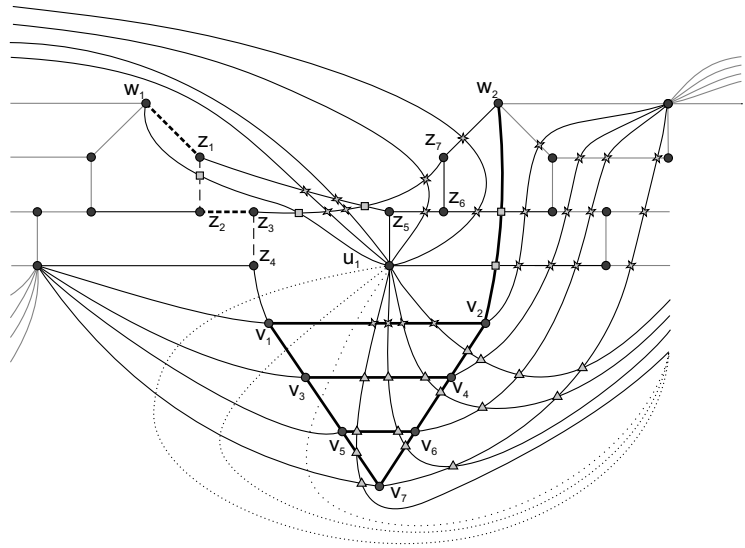}
\end{center}
\caption{A fragment of an optimal tile drawing of $G_{4,4}{}^{\updownarrow}$. Here $\ell=n=4$, which means that an optimal drawing has $45$ crossings: $\ell^2=16$ of triangles, $2\ell(n-1)=24$ of $4$-stars and $\binom{n}{2}-1=5$ of squares. For clearer presentation we use different $n$ compared to Figure~\ref{fig:CCEdges}. See Figure~\ref{fig:CriticalNew1} for additional explanation.}
\label{fig:CriticalNew4}
\end{figure}

To show that each edge of $H_{\ell, n}$ is critical, we first present $H_{\ell, n}$ with 4 types of edges (see Figure~\ref{fig:CCEdges}): thin solid, thin dashed, thick dashed and thick solid. Then we present four figures, where each figure focuses its attention on one special type of edges: Figure~\ref{fig:CriticalNew1} for thin solid edges, Figure~\ref{fig:CriticalNew2} for thin dashed edges, Figure~\ref{fig:CriticalNew3} for thick dashed edges and Figure~\ref{fig:CriticalNew4} for thick solid edges. Every edge of a given type is crossed in the appropriate figure. Sometimes, some edge $e$ is used in different optimal drawings to cross different ones of the edges of the specific type. The required optimal redrawings of such edges $e$ are indicated with dotted curves. 
\qed\end{proof}

%%%%%%%%%%%%%%%%%%%%%%%%%%%%%%%%%%%%%%%%%%%%%%%%%%%%%%%%%%%%%%%%%%%%%%%%%%% 

\section{Families with Prescribed Frequent Degrees}
\label{sc:freq}

We now get back to the primary question which motivated the research leading to
\cite{cit:newinfcrit} and this paper:
which vertex degrees other than $3,4,6$ can occur arbitrarily often in infinite
$k$-crossing-critical families?
First, we summarize the relevant particular constructions---our future building
blocks---obtained so far
(note that some of the claimed results have been proved in a more general
form than stated here, but we state them right in the form we shall use).

\begin{proposition}\label{prop:critconstructions}
There exist (infinite) families $\cF$ of simple, $3$-connected,
$k$-crossing-critical graphs such that, in addition, the following holds:
\begin{enumerate}[a)]
\item\label{it:even46}
(\cite[Section~4]{cit:newinfcrit}.)
For every $k\geq10$ or odd $k\geq5$, and every rational $r\in(4,6-\frac8{k+1})$,
a~family $\cF$ which is $\{4,6\}$-max-universal and each member of $\cF$ is of
average degree exactly~$r$, and another $\cF$ which is $\{4\}$-max-universal and of
average degree exactly~$4$.
Every graph of the two families has the set of its vertex degrees equal to
$\{3,4,6\}$ (e.g., degree $3$ repeats six times in each).
\item\label{it:even4high}
(\cite[Section~3~and~4]{cit:newinfcrit}.)
For every $\varepsilon>0$, any integer $k\geq5$ and every set $D_e$ of even integers
such that $\min(D_e)=4$ and $6\leq\max(D_e)\leq2k-2$,
a~family $\cF$ which is $D_e$-max-universal, and each graph of $\cF$
has the set of its vertex degrees $D_e\cup\{3\}$ and is of
average degree from the interval $(4,4+\varepsilon)$.
\item\label{it:3to3}
(\cite{cit:kochol} for $k=2$ and \cite{cit:avgcrit} for general $k$, see $G(0,n,m)$.)
For every $k=\binom{n}{2}-1$ where $n\geq3$ is an integer,
a family $\cF$ which is $\{3,4\}$-max-universal 
and each member of $\cF$ is of average degree equal to $3+\frac1{4n-7}$.
\item\label{it:anyodd}
($G(\ell,3,m)$ in Theorem~\ref{thm:constrodd}.)
For every $k=\ell^2+4\ell+2$ where $\ell\geq1$ is an integer,
a family $\cF$ which is $\{3,4,2\ell+3\}$-max-universal
and each member of $\cF$ is of average degree $5-\frac4{\ell+2}$.
\end{enumerate}
\end{proposition}

Having the particular constructions of
Proposition~\ref{prop:critconstructions} and the zip product
with Theorem~\ref{thm:zip3} at hand, it is now quite easy to give the 
``ultimate'' combined construction as follows.
For two graph families $\cF_1,\cF_2$ of simple $2$-connected graphs such that
each graph in $\cF_1\cup\cF_2$ has a vertex of degree~$3$,
we define the {\em zip product} of $\cF_1$ and $\cF_2$ as the family
of all graphs $H$ such that there exist $G_1\in\cF_1$, $G_2\in\cF_2$
and vertices $v_1\in V(G_1)$, $v_2\in V(G_2)$ of degree~$3$,
and $H$ is the zip product of $G_1$ and $G_2$ according to~$v_1,v_2$.

\begin{lemma}\label{lem:univzipsum}
Let $\cF_i$, $i=1,2$, be a $D_i$-max-universal family of simple
$2$-connected graphs such that each graph in $\cF_i$ has a vertex of
degree~$3$.
Then the zip product of $\cF_1$ and $\cF_2$ is a 
$(D_1\cup D_2)$-max-universal family.
\end{lemma}
\begin{proof}
Let $\cF$ denote the zip product of $\cF_1$ and $\cF_2$.
We first prove that $\cF$ is $(D_1\cup D_2)$-universal.
Choose any set of integers $\dset{m_d}{d\in D_1\cup D_2}$,
and, for $i=1,2$, graphs $G_i\in\cF_i$, such that, for each $d\in D_i \backslash \left\{3\right\}$, $G_i$ contains at least $m_d$
vertices of degree $d$,
and, if~$3\in D_i$, $G_i$ has at least $m_3+1$ vertices of degree $3$.
Then, for each $d\in D_1\cup D_2$, the zip product of $G_1$ and $G_2$ (according to any pair of their
degree-$3$ vertices) has at least $m_d$  
vertices of degree $d$.

Conversely, assume that $\cF$ is $\{d\}$-universal for some integer~$d$.
Then, for every integer $m$, 
there exists $G\in\cF$ such that $G$ has at least
$2m$ vertices of degree~$d$.
Since $G$ is a zip product of graphs $G_i\in\cF_i$, $i=1,2$,
one of $G_1,G_2$ contains at least $m$ vertices of degree~$d$.
W.l.o.g., this happens infinitely often for $i=1$, and so (up to symmetry)
$\cF_1$ is $\{d\}$-universal.
Therefore, $d\in D_1\cup D_2$ which proves that $\cF$ is $(D_1\cup
D_2)$-max-universal.
\qed\end{proof}

\begin{theorem}
\label{thm:alluniversal}
Let $D$ be any finite set of integers such that $\min(D)\geq3$. 
Then there is an integer $K=K(D)$, such that for every $k\ge K$, 
there exists a $D$-universal family of simple, $3$-connected,
$k$-crossing-critical graphs.
Moreover, if either $3,4\in D$ or both $4\in D$ and $D$ contains
only even numbers, then there exists a $D$-max-universal such family.
All the vertex degrees in the families are from $D\cup\{3,4,6\}$.
\end{theorem}

\begin{proof}
It suffices to prove the second claim ($D$-max-universal) since
a $(D\cup\{3,4\})$-max-universal family is also $D$-universal.
Furthermore, if $D$ contains only even numbers, then the claim has already
been proved in \cite{cit:newinfcrit}, here in
Proposition~\ref{prop:critconstructions}\,\ref{it:even4high}).

Hence assume the case $3,4\in D$, and let $D_e\subseteq D$
be the subset of the even integers from~$D$.
Let $\cF_e$ denote the family from
Proposition~\ref{prop:critconstructions}\,\ref{it:even4high}) with
$k_e=\frac12\max(D_e)+1$, and $\cF_3$ the family from
Proposition~\ref{prop:critconstructions}\,\ref{it:3to3}) with~$k=2$.
For every $a\in D\setminus D_e$, $a>3$, let $\cF_a$ denote the family from                                          
Proposition~\ref{prop:critconstructions}\,\ref{it:anyodd}) with~$2\ell_a+3=a$
and crossing number $k_a=\ell_a^2+4\ell_a+2$.
Since, in particular, $\cF_3$ is $\{3\}$-universal, 
we may assume that every graph in $\cF_3$ has more than $|D\setminus D_e|$
vertices of degree~$3$.
We now construct a family $\cF$ as the iterated zip product of $\cF_3$,
$\cF_e$, and (possibly) of each $\cF_a$ where $a\in D\setminus D_e$, $a>3$.

Clearly, every graph from $\cF$ is simple and $3$-connected.
By Lemma~\ref{lem:univzipsum}, $\cF$ is $D$-max-universal,
and by Theorem~\ref{thm:zip3}, $\cF$ is $K$-crossing-critical
where $K=k_e+2+\sum_{a\in D\setminus D_e,a>3}k_a$.
This construction creates only vertices of degrees from $D\cup\{3,4,6\}$.
To extend the construction of $\cF$ to any parameter $k>K$, 
we simply replace the family $\cF_e$ by analogous $\cF_e'$ from
Proposition~\ref{prop:critconstructions}\,\ref{it:even4high})
with the parameter $k_e'=k_e+(k-K)$.
\qed\end{proof}

\begin{figure}[th]
\begin{center}
\vspace*{-3mm}%
\includegraphics[width=0.9\hsize]{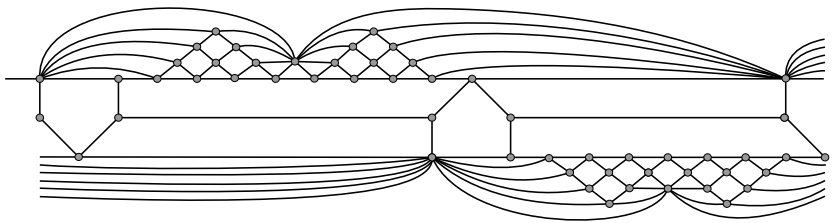}
\end{center}
\caption{A possible (alternative) way of combining the ideas of the construction
	\cite{cit:newinfcrit} with the tile $G_{5,3}$.}
\label{fig:NewGraphG}
\end{figure}
\vspace*{-3mm}%
At last we shortly remark that building blocks of the ``crossed belt'' construction of
\cite{cit:newinfcrit}
(Proposition~\ref{prop:critconstructions}\,\ref{it:even4high}) can be directly combined
with the new construction of $G(\ell,n,m)$, without invoking a zip product.
Such a combination is outlined in Figure~\ref{fig:NewGraphG}.
However, since this construction can only achieve a combination of various
even degrees with one prescribed odd degree (greater than~$3$), it cannot
fully replace the proof of Theorem~\ref{thm:alluniversal} and so we refrain from
giving the lengthy technical details in this paper.

%%%%%%%%%%%%%%%%%%%%%%%%%%%%%%%%%%%%%%%%%%%%%%%%%%%%%%%%%%%%%%%%%%%%%%%%%%% 

\section{Families with Prescribed Average Degree}
\label{sc:average}

In addition to Theorem~\ref{thm:alluniversal}, we are going to show that
the claimed $D$-max-universality property can be combined with nearly any
feasible rational average degree of the family.

\begin{theorem}\label{thm:alluniversal-avgdeg}
Let $D$ be any finite set of integers such that $\min(D)\geq3$
and $A\subset\mathbb R$ be an interval of reals.
Assume that at least one of the following assumptions holds:
\begin{enumerate}[a)]\parskip0pt
\item\label{it:avg346}
$D\supseteq\{3,4,6\}$ and $A=(3,6)$,
\item\label{it:avg34}
$D\supsetneq\{3,4\}$ and $A=(3,4]$,
or $D=\{3,4\}$ and $A=(3,4)$,
\item\label{it:avg35}
$D\supsetneq\{3,4\}$ and $A=(3,5-\frac8{b+1})$ where $b$ is the largest odd
number in~$D$ and $b\geq9$ (note that $b=7$ is covered in (\ref{it:avg34})),
\item\label{it:avgeven}
$D\supseteq\{4,6\}$ contains only even numbers and $A=(4,6)$, or $D=\{4\}$ and $A=\{4\}$.
\end{enumerate}
Then, for every rational $r\in A$,
there is an integer $K=K(D,r)$ such that for every $k\ge K$, 
there exists a $D$-max-universal family of simple $3$-connected
$k$-crossing-critical graphs of average degree precisely~$r$.
\end{theorem}

Before we prove the theorem, we informally review the coming steps.
The basic idea of balancing the average degree in a crossing-critical family
is quite simple; assume we have two families $\cF_a,\cF_b$ of fixed average
degrees $a<b$, respectively, and containing some degree-$3$ vertices.
Then, we can use zip product of graphs from the two families to obtain new
graphs of average degrees which are convex combinations of $a$ and $b$.
This simple scheme, however, has two difficulties:
\begin{itemize}
\item If one combines graphs $G_1\in\cF_a$ and $G_2\in\cF_b$,
then the average degree of the disjoint union $G_1\cup G_2$
is the average of $a,b$ weighted by the sizes of $G_1,G_2$.
Hence we need great flexibility in choosing members of $\cF_a,\cF_b$
of various size, and this will be taken care of by the notion of a 
{\em scalable family}.
\item Second, after applying a zip product of $G_1,G_2$ the resulting
average degree is no longer this weighted average of $a,b$ but a slightly
different rational number.
We will take care of this problem by introducing a special {\em compensation
gadget} whose role is to revert the change in average degree caused by the zip
product.
\end{itemize}

\begin{figure}[th]
\begin{center}
\includegraphics[width=0.5\hsize]{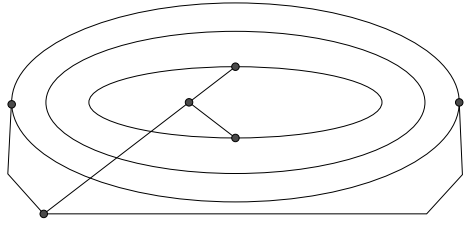}
\end{center}
\caption{The $k$-crossing-critical ``crossed belt'' construction of \cite{cit:newinfcrit}:
	the shaded part is any plane graph consisting of an edge-disjoint
	union of $k$ cycles, satisfying certain (rather weak) technical
	and connectivity conditions;
	the six marked vertices are all of degree three.}
\label{fig:crossedbelt}
\end{figure}

\begin{figure}[th]
\begin{center}
\includegraphics[width=0.25\hsize]{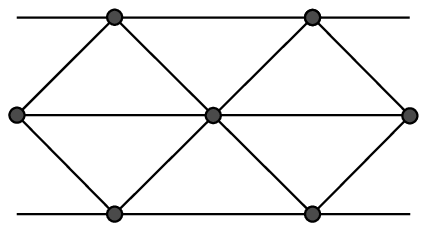}
\qquad\raise6ex\hbox{\Large$\leadsto$}\qquad
\includegraphics[width=0.25\hsize]{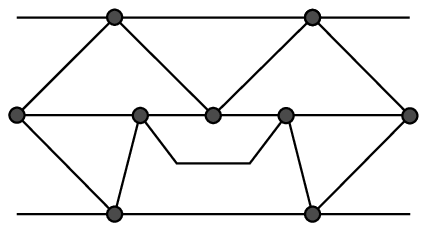}
\end{center}
\caption{The tile $T$ (left) used to construct our ``compensation gadget''
	$M_m$, and the tile $T''$ (called ``double-split'' in
	\cite{cit:newinfcrit}) that can replace $T$ in the compensation
	gadget.}
\label{fig:compensation}
\end{figure}

We start with addressing the second point.
The {\em compensation gadget} (one for a whole family) will be picked from
the family in Proposition~\ref{prop:critconstructions}\,\ref{it:even46}).
To describe it precisely, we have to (at least informally) introduce the
very general {\em crossed belt} construction of crossing-critical families
from~\cite{cit:newinfcrit}---see it is Figure~\ref{fig:crossedbelt}.
Let $T$ be the planar tile depicted in Figure~\ref{fig:compensation} on the
left, and let $M_m'$ be the planar graph obtained as the cyclization
$\circ(T_0,\dots,T_{m-1})$ where each $T_i=T$.
Let $M_m$, $m\geq12$, be constructed from $M_m'$ by adding six new degree-$3$ vertices 
and five new edges as in Figure~\ref{fig:crossedbelt}, such that four of
the new vertices subdivide rim edges of the tiles
$T_0,T_{\lfloor m/4\rfloor},T_{\lfloor m/2\rfloor},T_{\lfloor3m/4\rfloor}$.
Let $M_m^c$ be constructed exactly as $M_m$ but replacing arbitrary $c\geq0$ of
the tiles $T$ with $T''$ shown on the right in Figure~\ref{fig:compensation}.

\begin{proposition}[\cite{cit:newinfcrit}]
For any $m\geq12$ and $0\leq c\leq m$, the graph $M_m^c$ is $5$-crossing-critical.
\end{proposition}

The way ``compensating by'' the gadget $M_m^c$ works, is formulated next.

\begin{lemma}\label{lem:compensatezip}
Let $G_1,\dots,G_t$ be graphs, each having at least two degree-$3$ vertices,
and let $q\in\mathbb N$ and $m\geq\max(q+t,12)$.
If $H$ is a graph obtained using the zip product of all $G_1,\dots,G_t$ and of
$M_m^{q+t}$ (in any order and any way), 
then the average degree of
$H$ is equal to the average degree of the disjoint union of $G_1,\dots,G_t$
and~$M_m^q$.
\end{lemma}
\begin{proof}
Let $n_i=|V(G_i)|$ and $s_i$ be the sum of degrees in~$G_i$, and let
$n_0=6m+6+2q$, $s_0=28m+18+6q$ be the same quantities in~$M_m^q$.
Then $n_0''=|V(M_m^{q+t})|=n_0+2t$ and the sum of degrees of $M_m^{q+t}$ is
$s_0''=s_0+6t$.
Since performing one zip operation decreases the number of vertices by $2$
and the sum of degrees by $6$, we have
$|V(H)|=n_0''+n_1+\dots+n_t-2t=n_0+n_1+\dots+n_t$
and the sum of degrees in $H$ is $s_0''+s_1+\dots+s_t-6t=s_0+s_1+\dots+s_t$,
and the claim follows.
\qed\end{proof}

To address the first point, we give the following definition.
A family of graphs $\cF$ is {\em scalable} if all the graphs in $\cF$ have
equal average degree and for every $G\in\cF$ and every integer $a$,
there exists $H\in\cF$ such that $|V(H)|=a|V(G)|$.
Furthermore, $\cF$ is {\em$D$-max-universal scalable} if, additionally,
$H$ contains at least $a$ vertices of each degree from $D$
and the number of vertices of degrees not in $D$ is bounded from above 
independently of~$a$.

Trivially, the families of
Proposition~\ref{prop:critconstructions}\,\ref{it:3to3}),\ref{it:anyodd})
are $D$-max-universal scalable for $D=\{3,4\}$ and $D=\{3,4,2\ell+3\}$,
respectively.
For the families as in
Proposition~\ref{prop:critconstructions}\,\ref{it:even46}),\ref{it:even4high}),
we have:

\begin{lemma}\label{lem:makescalable}
There exist families, satisfying the conditions of
Proposition~\ref{prop:critconstructions}\,\ref{it:even46}) and
\ref{it:even4high}), respectively,
which are $D$-max-universal scalable for their respective sets~$D$.
\end{lemma}

Note that in the previous case of 
Proposition~\ref{prop:critconstructions}\,\ref{it:even4high}), 
the claimed family from~\cite{cit:newinfcrit} 
was not required to have fixed average degree.
Now also the family extending this case~\ref{it:even4high}) will be of
fixed average degree.
\begin{proof}
The proof is completely based on the constructions from
\cite{cit:newinfcrit}, but since the question of scalability is not
considered there, we have to discuss some further details of the
crossed belt construction of \cite{cit:newinfcrit} (recall Figure~\ref{fig:crossedbelt}).

First, consider a $\{4\}$-max-universal family $\cF_4$
of simple, $3$-connected, $k$-crossing-critical graphs of average degree~$4$,
as in Proposition~\ref{prop:critconstructions}\,\ref{it:even46}).
Pick any $G\in\cF_4$; then $G$ has precisely six degree-$3$ vertices, and since the
only other vertex degrees occurring in $G$ are $4$ and $6$, $G$ has precisely
three degree-$6$ vertices.
Let $G'$ be the ``planar belt'' of $G$ (the shaded part in
Figure~\ref{fig:crossedbelt}, without degree-$3$ vertices).
Then $G'$ can be cut to form a perfect planar tile $T_{G'}$ such that
$\circ\,T_{G'}=G'$.
For integer $a\geq1$, let $G'_a$ denote the cyclization of $a$ copies of
$T_{G'}$, and let $G''_a$ denote the graph $G'_a$ with the six degree-$3$
vertices added back (such that four of them subdivide the same edges of one
copy of the tile $T_{G'}$ as they do in~$G$).
By \cite{cit:newinfcrit}, $G''_a$ is again $k$-crossing-critical.
If $n=|V(G')|$ and $s$ is the degree sum of $G'$, then
$|V(G)|=n+6$ and the degree sum of $G$ is $s+18$.
Furthermore, $|V(G''_a)|=an+6$ and the degree sum of $G''_a$ is $as+18$,
and $G''_a$ has $3a$ degree-$6$ vertices.
We denote by $G_a$ the graph obtained by $3a-3$ ``double split'' operations
each replacing a degree-$6$ vertex by three degree-$4$ vertices as
illustrated in Figure~\ref{fig:compensation}.
Then $|V(G_a)|=an+6+2(3a-3)=a|V(G)|$ and the degree sum of $G_a$ is
$as+18+6(3a-3)=a(s+18)$, and so the average degree is the same as of~$G$.
There are only three degree-$6$ vertices left in~$G_a$.
Hence, for every $a>1$, we may assume $G_a\in\cF_4$ as well.

Second, consider a $\{4,6\}$-max-universal family $\cF_r$
of simple $3$-connected $k$-crossing-critical graphs of average degree
$r\in(4,6-\frac8{k+1})$,
as in Proposition~\ref{prop:critconstructions}\,\ref{it:even46}).
Then the proof follows the same line as in the previous paragraph,
only that now we have many degree-$6$ vertices by
the assumption of $\{6\}$-universality.

Third, consider a $D_e$-max-universal family $\cF_e$
of simple $3$-connected $k$-crossing-critical graphs,
as in Proposition~\ref{prop:critconstructions}\,\ref{it:even4high}).
This case is somehow different from the previous two since we have no
vertices of degree~$6$ (unless $6\in D_e$) and $\cF_e$ contains graphs of
various average degrees.
Though, for any fixed~$\varepsilon>0$, $\cF_e$ can be chosen 
% by Proposition~\ref{prop:critconstructions}\,\ref{it:even4high})
such that the average degree of every member of $\cF_e$ is from the interval
$(4,4+\varepsilon/2)$.
Pick arbitrary but sufficiently large $G\in\cF_e$.
Then one can find (see \cite{cit:newinfcrit} for details)
three edges in $G$ not close to each other and not having vertices of degree 
other than $4$ in close neighbourhood, and let $G_1$ be obtained by
contracting these three edges (into vertices of degree~$6$).
By \cite{cit:newinfcrit}, $G_1$ is again $k$-crossing-critical.
Since $G$ is sufficiently large, the average degree of $G_1$ is equal to some
$r_1\in(4,4+\varepsilon)$.
Now the construction from the first case above applies to~$G_1$ and gives
a whole scalable family of average degree~$r_1$.
\qed\end{proof}

The next step is to combine suitable scalable families to obtain arbitrary
rational average degrees in a given interval (roughly, between the sparsest
and the densest available family).

\begin{lemma}\label{lem:composedegrees}
Assume, for $i=1,\dots,t$, that $\cF_i$ is a $D_i$-max-universal scalable
family of simple $3$-connected $k_i$-crossing-critical graphs of average
degree exactly~$r_i$, and that every graph in $\cF_1\cup\dots\cup\cF_t$ 
has at least two degree-$3$ vertices.
For every $k\geq k_1+\dots+k_t+5$, there exists rational $r_0\in(3,6)$, such
that the following holds for every $a_1,\dots,a_t,c\in\mathbb N$:
\begin{enumerate}[a)]\parskip0pt
\item\label{it:degreefreq}
there exists a simple, $3$-connected, $k$-crossing-critical graph $G$
having at least $a_i$ vertices of each degree from $D_i$,
\item\label{it:degreefew}
the number of vertices of $G$ of degree not in $D_1\cup\dots\cup D_t$ is
bounded from above by a number depending only on $c,k$ 
and the families $\cF_1,\dots,\cF_t$, and
\item the average degree of $G$ is precisely
\begin{equation}\label{eq:avgclaim}
r=\frac{\sum_{i=1}^ta_ir_i+cr_0}{\sum_{i=1}^ta_i+c} 
\,.\end{equation}
\end{enumerate}
\end{lemma}
\begin{proof}
Let $\ell=k-(k_1+\dots+k_t+5)$ and denote by $\cK_\ell$ a set of $\ell$
disjoint copies of the graph $K_{3,3}$.
Pick arbitrary $G_i\in\cF_i$, $i=1,\dots,t$.
We may w.l.o.g.~assume that $n_0=|V(G_1)|=\dots=|V(G_t)|$
and $n_0$ divisible by~$6$,
since otherwise we take the least common multiple of $6$ and all the graph sizes
and apply scalability of the families~$\cF_i$.
Clearly, $n_0$ can be chosen arbitrarily large as well, such as
$n_0\geq6(4\ell+t+4)$.
Let $G_0=M_{n_0/6-(\ell+1)}^0$ (the compensation gadget defined above)
and $H_0$ denote the disjoint union of $\cK_\ell$ and $G_0$.
Then $|V(H_0)|=n_0$ and we choose $r_0$ to be the average degree of~$H_0$;
$$
r_0=\frac{18\ell+28(\frac{n_0}{6}-(\ell+1))+18}{n_0}
   =\frac{\frac{14n_0}{3}-10(\ell+1)}{n_0}
\,.$$

Again by scalability, for $i=1,\dots,t$, there exist $G_i^{\,*a_i}\in\cF_i$ (of average degree
$r_i$) such that
$|V(G_i^{\,*a_i})|=a_in_0$.
Similarly, we let $G_0^{\,*c}=M_{cn_0/6-c(\ell+1)}^{3(c-1)(\ell+1)}$.
It is simple calculus to verify that the disjoint union of $\cK_\ell$ and
$G_0^{\,*c}$ has $cn_0$ vertices and the average degree equal to
$$
\frac{18\ell+28(\frac{cn_0}{6}-c(\ell+1))+18(c-1)(\ell+1)+18}{cn_0}
   =\frac{\frac{14cn_0}{3}-10c(\ell+1)}{cn_0}=r_0
\,.$$
Hence the average degree of the disjoint union of $\cK_\ell$ and
$G_0^{\,*c}$ and $G_1^{\,*a_1},\ldots, G_t^{\,*a_t}$ indeed is
\begin{equation}\label{eq:avgcompens}
\frac{\sum_{i=1}^t a_in_0r_i + cn_0r_0}{\sum_{i=1}^ta_in_0+cn_0} = r 
\,.\end{equation}

Finally, we let ${G'_0}^{*c}=M_{cn_0/6-c(\ell+1)}^{3(c-1)(\ell+1)+\ell+t}$
and construct the simple $3$-connected graph $G$ 
as the zip product of $\cK_\ell$ and
${G'_0}^{*c}$ and $G_1^{\,*a_1},\dots,G_t^{\,*a_t}$.
By Theorem~\ref{thm:zip3}, $G$ is $k$-crossing-critical with $k=\ell+5+k_1+\dots+k_t$, as required.
The degrees condition in \ref{it:degreefreq}) follows from
max-universal scalability of $\cF_1,\dots,\cF_t$,
and \ref{it:degreefew}) then follows as well since the size of ${G'_0}^{*c}$
is bounded with respect to~$c,k$.
Moreover, by compensation Lemma~\ref{lem:compensatezip},
the average degree of $G$ is equal to~$r$, as in \eqref{eq:avgcompens}.
\qed\end{proof}

\begin{corollary}[Lemma~\ref{lem:composedegrees}]\label{cor:composedegrees}
Assume $D_i$-max-universal scalable $k_i$-cross\-ing-critical
families $\cF_i$ of average degree $r_i$, $i=1,\dots,t$, 
as in Lemma~\ref{lem:composedegrees}, such that $r_1<r_2$.
Then, for every $k\geq k_1+\dots+k_t+5$ and every $r\in(r_1,r_2)\cap\mathbb Q$, 
there exists a $(D_1\cup\dots\cup D_t)$-max-universal
family of simple, $3$-connected, $k$-crossing-critical graphs of average
degree exactly~$r$.
\end{corollary}

\begin{proof}
The proof is a simple exercise in calculus based on Lemma~\ref{lem:composedegrees}.
Let $r=\frac pq$ where $p,q$ are relatively prime integers.
Our task is to find infinitely many suitable choices of $a_1,\dots,a_t$ 
such that, by \eqref{eq:avgclaim},
\begin{equation}\label{eq:togetpq}
\frac pq=\frac{\sum_{i=1}^ta_ir_i+cr_0}{\sum_{i=1}^ta_i+c}
\end{equation}
for some (unknown) rational $r_0\in(3,6)$ and suitable (but fixed, see below)~$c$.

To further simplify the task, we choose sufficiently large
integer $m$ such that $r_1'=(mr_1+r_3+\dots+r_{t})/(m+t-2)<r$
and set $a_1=ma$, $a_3=\dots=a_{t}=a$, $a_2=b$ for yet unknown~$a,b$.
Then \eqref{eq:togetpq} reads:
$$
\frac pq=\frac{mar_1+ar_3+\dots+ar_{t}+br_2+cr_0}{a(m+t-2)+b+c}
  =\frac{a(m+t-2)r_1'+br_2+cr_0}{a(m+t-2)+b+c}
$$
Let $s=m+t-2$, and $r_1'=\frac{p_a}{q_a}$, $r_2=\frac{p_b}{q_b}$, $r_0=\frac{p_0}{q_0}$.
We continue with equivalent processing:
$$
\frac pq = \frac{as\frac{p_a}{q_a} +b\frac{p_b}{q_b} +c\frac{p_0}{q_0}}{as+b+c}
$$ $$
p(as+b+c)q_aq_bq_0 = asqq_bq_0p_a+bqq_aq_0p_b+cqq_aq_bp_0
$$
Finally, we get that \eqref{eq:togetpq} under our special substitution for
$a_1,\dots,a_t$, is equivalent to the following
linear Diophantine equation in $a,b$:
$$
a\cdot sq_bq_0(pq_a-p_aq) + b\cdot q_aq_0(pq_b-p_bq) =
	cq_aq_b(p_0q-pq_0)
$$
Setting $c=q_0\cdot GCD\big(sq_b(pq_a-p_aq),q_a(pq_b-p_bq)\big)$,
this equation has infinitely many integer solutions, and since $r_1'<r<r_2$,
we have that $pq_a-p_aq>0$ and $pq_b-p_bq<0$ and so infinitely many of the
solutions are among positive integers
(regardless of whether the right-hand side is positive, zero or negative).
\qed\end{proof}

\begin{proof}[Proof of Theorem~\ref{thm:alluniversal-avgdeg}]
The case \ref{it:avgeven}) has already been proved in \cite{cit:newinfcrit},
see Proposition~\ref{prop:critconstructions}\,\ref{it:even46}).
In all other cases, let $\cF_1$ be the family from
Proposition~\ref{prop:critconstructions}\,\ref{it:3to3}) such that the
parameter $n$ satisfies $r_1=3+\frac1{4n-7}<r$ (where $r\in A\cap\mathbb Q$,
$r>3$, is the desired fixed average degree).

In the case \ref{it:avg346}), let $\cF_2$ be a family from
Proposition~\ref{prop:critconstructions}\,\ref{it:even46})
with average degree equal to arbitrary (but fixed)
$r_2\in(r,6)\not=\emptyset$, which can be chosen as scalable by
Lemma~\ref{lem:makescalable}.
In the case~\ref{it:avg35}), let $\cF_2$ be the family from
Proposition~\ref{prop:critconstructions}\,\ref{it:anyodd})
with the parameter $\ell$ such that $b=2\ell+3$;
in this case $r_2=5-\frac8{b+1}>r$.
Finally, we consider the remaining subcases of~\ref{it:avg34}).
If $D=\{3,4\}$, then let $\cF_2$ be the second family from
Proposition~\ref{prop:critconstructions}\,\ref{it:even46})
with average degree~$r_2=4$.
If $D\supsetneq\{3,4\}$, then let $\cF_2$ be the family from
Proposition~\ref{prop:critconstructions}\,\ref{it:even4high}), made scalable
and of fixed average degree $r_2>4$ by Lemma~\ref{lem:makescalable}.

In each one of the choices of $\cF_1,\cF_2$ above, $r_1<r<r_2$ holds.
Furthermore, if necessary in order to fulfil $D$-max-universality, we
introduce additional scalable families $\cF_3,\dots$ as in the proof of
Theorem~\ref{thm:alluniversal}.
Then, Theorem~\ref{thm:alluniversal-avgdeg} follows directly from
Corollary~\ref{cor:composedegrees}.
\qed\end{proof}

\section{Degree Properties in 2-Crossing-Critical Families}
\label{sc:2cc} 

In the previous constructions, we have always assumed that the fixed 
crossing number $k$ of the families is sufficiently large.
One can, on the other hand, ask what happens if we fix a (small) 
value of~$k$
beforehand (i.e., independently of the asked degree properties).

\begin{figure}[!ht]
\begin{center}
\includegraphics[width=0.35\hsize]{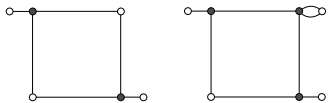}
\end{center}
\caption{The two frames.}
\label{fig:2crossframes}
\end{figure}
\vspace*{-6mm}%
\begin{figure}[!ht]
\begin{center}
\includegraphics[width=0.65\hsize]{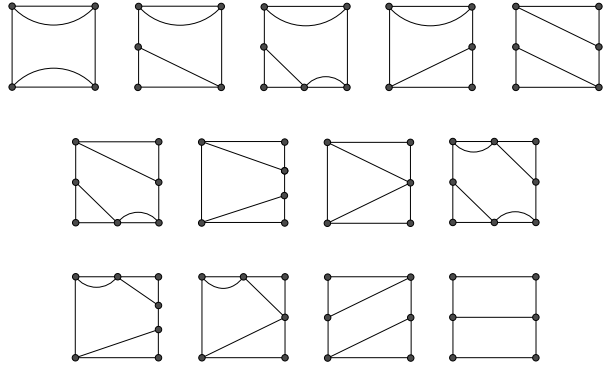}
\end{center}
\caption{The thirteen pictures for Definition~\ref{df:2critchar}.}
\label{fig:2crosspictures}
\end{figure}

In this direction, there is the remarkable result of  Dvo\v{r}\'ak and
Mohar~\cite{cit:dvorakmohar} proving the existence of $k$-crossing-critical
families with unbounded maximum degree for any $k\geq 171$.
Unfortunately, since \cite{cit:dvorakmohar} is not really constructive, we
do not know anything exact about the degrees occurring in these families.
An explicit construction of a $k$-crossing-critical family with unbounded
maximum degree is known only in the projective plane~\cite{cit:starsbonds}
for $k\geq2$, but that falls outside of the area of interest of this paper.

It thus appears natural to thoroughly investigate the least non-trivial case
of~$k=2$, with significant help of the characterization
result~\cite{cit:2critchar}.
In a nutshell, \cite{cit:2critchar} claims that nearly all
$2$-crossing-critical graphs are built from a certain rather small finite 
set of tiles.
The formal result is stated next.

We refer to Section~\ref{sc:tools} for definitions of the 
operations $\circ$, $\otimes$ and $^{\updownarrow}$ on tiles.

\begin{definition}
\label{df:2critchar}
Let $\cS$ denote the set of tiles which are 
obtained as combinations of one of the two frames, 
illustrated in Figure \ref{fig:2crossframes},
and the $13$ pictures, shown in Figure \ref{fig:2crosspictures}.
These are combined in a way that 
a picture is inserted into a frame by identifying
the two geometric squares (typically by subdividing some edges of the frame).
A given picture may be inserted into a frame either with the given orientation 
or with a $180^\circ$ rotation. Note that $\cS$ contains 42 different tiles.

We inductively define the following:
The set of \DEF{odd tiles} $\cT_o (\cS)$ consists of 
all the tiles that are either in $\cS$
or are obtained as $T_e\otimes T$ with $T_e\in\cT_e$ and $T\in\cS$.
The set of \DEF{even tiles} $\cT_e(\cS)$ 
consists precisely of the tiles obtained as 
$T_o\oplus {}^\updownarrow T^\updownarrow$, where $T_o\in\cT_o(\cS)$ 
and $T\in\cS$. 
Note that an odd number of tiles of $\cS$ is used to 
construct a tile of $\cT_o(\cS)$, and an even number for $\cT_e(\cS)$.

Let the set $\cG(\cS)$ consist precisely of all the graphs of the form 
$G=\circ(T)$ where $T\in\cT_o(\cS)\setminus \cS$. 
Note that each graph of $\cG(\cS)$ is obtained as $G=\circ (\cT^{\updownarrow})$, 
where $\cT$ is a sequence $(T_0, ^{\updownarrow}{T_1}^{\updownarrow}, T_2, \dots, ^{\updownarrow}{T_{2m-1}}^{\updownarrow}, T_{2m})$
such that $m\geq 1$ and $T_i \in \cS$ for each $i = 0, 1, 2, \dots, 2m$.
\end{definition}
\begin{figure}[h]
\centering
\subfigure[tile $T_a$]{\includegraphics[width=0.15\textwidth]{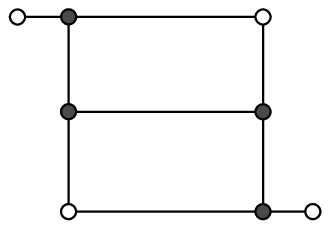}}
\qquad
\subfigure[tile $T_b$]{\includegraphics[width=0.15\textwidth]{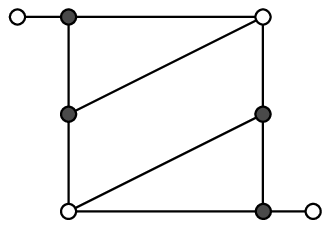}}
\qquad
\subfigure[tile $T_c$]{\includegraphics[width=0.15\textwidth]{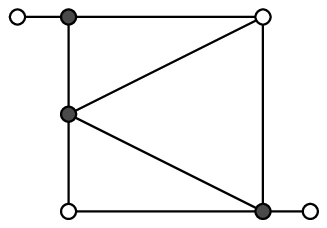}}
\qquad
\subfigure[tile $T_d$]{\includegraphics[width=0.15\textwidth]{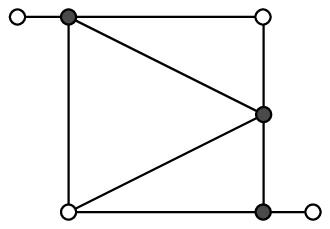}}
\quad

\subfigure[tile $T_e$]{\includegraphics[width=0.15\textwidth]{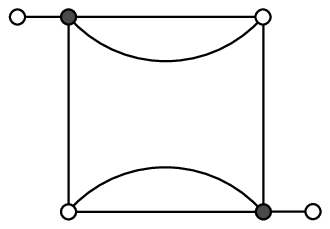}}
\qquad
\subfigure[tile $T_f$]{\includegraphics[width=0.15\textwidth]{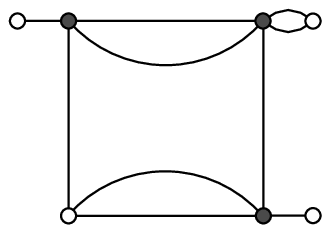}}
\qquad
\subfigure[tile $T_g$]{\includegraphics[width=0.15\textwidth]{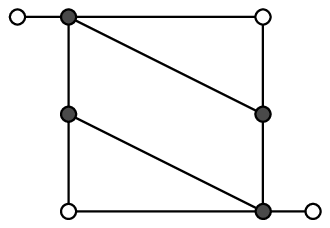}}
\qquad
\subfigure[tile $T_h$]{\includegraphics[width=0.15\textwidth]{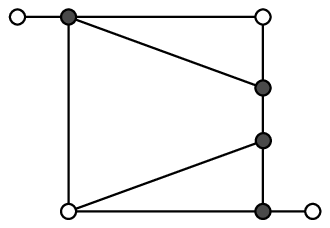}}
\quad

\subfigure[tile $T_i$]{\includegraphics[width=0.15\textwidth]{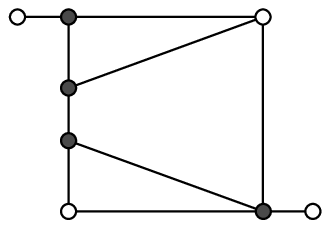}}
\qquad
\subfigure[tile $T_j$]{\includegraphics[width=0.15\textwidth]{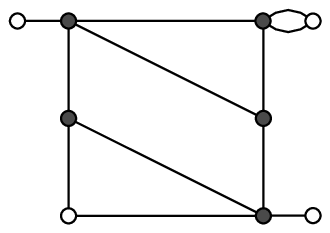}}
\qquad
\subfigure[tile $T_k$]{\includegraphics[width=0.15\textwidth]{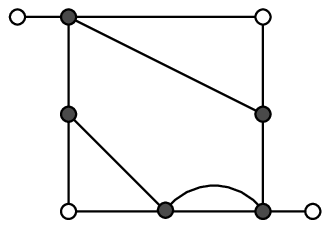}}
\qquad
\subfigure[tile $T_l$]{\includegraphics[width=0.15\textwidth]{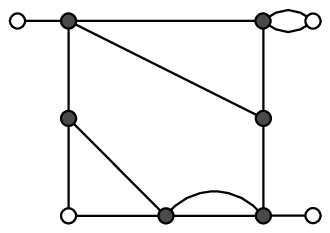}}
\quad

\subfigure[tile $T_m$]{\includegraphics[width=0.15\textwidth]{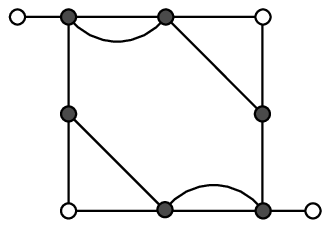}}
\qquad
\subfigure[tile $T_n$]{\includegraphics[width=0.15\textwidth]{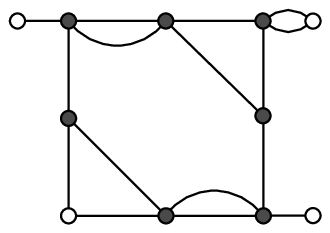}}
\qquad
\subfigure[tile $T_o$]{\includegraphics[width=0.15\textwidth]{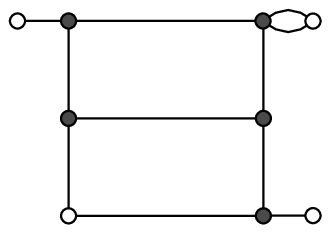}}
\qquad
\subfigure[tile $T_p$]{\includegraphics[width=0.15\textwidth]{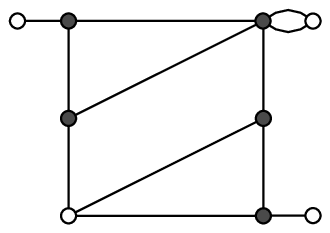}}
\caption{Examples of tiles from $\cS$.}
\label{fig:2crossexamples}
\end{figure}

Some examples of tiles from $\cS$, hereafter named
from $T_a$ to $T_e$, are shown in Figure~\ref{fig:2crossexamples}. 
We will use these tiles from $T_a$ to $T_p$ in what follows.

\begin{theorem}(\cite{cit:2critchar})
\label{th:2critchar}
There exist only finitely many $3$-connected $2$-crossing-critical graphs
which do not contain a subdivision of the graph $V_{10}$,
which is obtained from a $10$-cycle by adding all the $5$ main diagonals.

Then, $G$ is a $3$-connected $2$-crossing-critical graph containing a subdivision 
of $V_{10}$, if and only if $G \in \cG(\cS)$.
\end{theorem}

Since we are interested exclusively in infinite families of $2$-crossing-critical
graphs, we can focus on the graphs in $\cG(\cS)$, 
as any remaining (necessarily finite) subset of graphs disjoint from $\cG(\cS)$
would not affect degree properties of our families.
Note also that any $3$-connected, $2$-crossing-critical family of graphs contains at most 
finitely many graphs which are not almost-planar 
because any graph from $\cG(\cS)$ is almost-planar.

\begin{theorem}
\label{th:2crossD}
A $3$-connected $2$-crossing-critical $D$-max-universal family 
\begin{enumerate}[a)]
\item\label{it:a_simple}
of simple graphs exists if and only if $\{3\}\subsetneq D \subseteq \{3,4,5,6\}$.
	\item\label{it:b_general} exists if and only if $D \subseteq \{3,4,5,6\}$, $|D|\ge 2$, and $D\cap\{3,4\}\neq \emptyset$.
\end{enumerate}

\end{theorem}

\begin{proof}
Let $D$ be a set of positive integers and let $\cF$ be any $3$-connected $2$-crossing-critical $D$-max-universal family. By Theorem \ref{th:2critchar}, we may assume $\cF\subseteq\cG(\cS)$.\\
For case \ref{it:a_simple}), there are only nine simple tiles in $\cS$, and by join of any two of them 
we can only construct vertices with degrees 3, 4, 5 and 6, so 
$D\subseteq \{3,4,5,6\}$.
On the other hand, any simple tile from $\cS$ has a vertex of 
degree 3 that 
is not in its left or right wall, so $\{3\}\subseteq D$,
and we get some other vertex with degree not equal to 3 after we
join any two of them, so $\{3\}\subsetneq D$.

Now it only remains to construct a family $\cF$ for a set $D$
such that $\{3\}\subsetneq D \subseteq \{3,4,5,6\}$.
To this end, consider the following sequences of length $2m+1$:
$$
\aligned
\cT(\{3,4\},m)&=(T_a, ^{\updownarrow}{T_a}^{\updownarrow}, T_a, \dots, ^{\updownarrow}{T_a}^{\updownarrow}, T_a)\quad \text{(Figure \ref{fig:2cross34})},\\
\cT(\{3,5\},m)&=(T_a, ^{\updownarrow}{T_b}^{\updownarrow}, T_a, \dots, ^{\updownarrow}{T_b}^{\updownarrow}, T_a)\quad \text{(Figure \ref{fig:2cross35})},\\
\cT(\{3,6\},m)&=(T_b, ^{\updownarrow}{T_b}^{\updownarrow}, T_b, \dots, ^{\updownarrow}{T_b}^{\updownarrow}, T_b)\quad \text{(Figure \ref{fig:2cross36})},\\
\cT(\{3,4,5\},m)&=(T_a, ^{\updownarrow}{T_d}^{\updownarrow}, T_a, \dots, ^{\updownarrow}{T_d}^{\updownarrow}, T_a)\quad \text{(Figure \ref{fig:2cross345})},\\
\cT(\{3,4,6\},m)&=(T_c, ^{\updownarrow}{T_d}^{\updownarrow}, T_c, \dots, ^{\updownarrow}{T_d}^{\updownarrow}, T_c)\quad \text{(Figure \ref{fig:2cross346})},\\
\cT(\{3,4,5,6\},m)&=(T_c, ^{\updownarrow}{T_b}^{\updownarrow}, T_c, \dots, ^{\updownarrow}{T_b}^{\updownarrow}, T_c)\quad \text{(Figure \ref{fig:2cross3456})}.
\endaligned
$$

\begin{figure}[t]
\psfrag{Lww}{$L_{ww}$}
\psfrag{u1}{$\sigma$}
\psfrag{wl}{$\alpha$}
\centering
\begin{minipage}{.48\textwidth}
	\centering\includegraphics[width=0.9\textwidth]{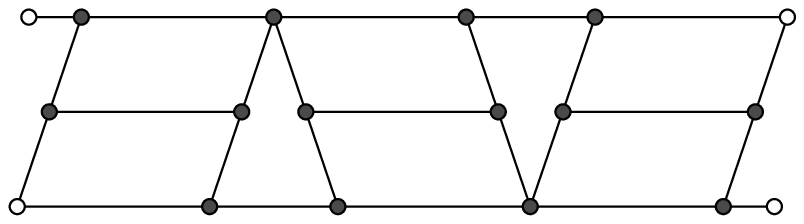}
	\caption{Tile $\otimes\cT(\{3,4\},m)$ for $m=1$.}\label{fig:2cross34}
\end{minipage}
\begin{minipage}{.48\textwidth}
	\centering\includegraphics[width=0.9\textwidth]{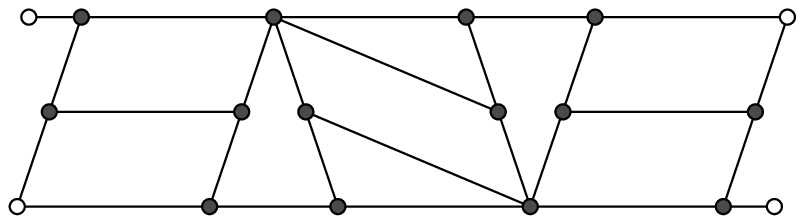}
	\caption{Tile $\otimes\cT(\{3,5\},m)$ for $m=1$.}\label{fig:2cross35}
\end{minipage}

\vspace*{3mm}%
\begin{minipage}{.48\textwidth}
	\centering\includegraphics[width=0.9\textwidth]{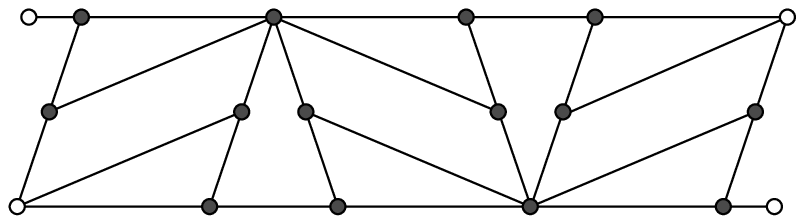}
	\caption{Tile $\otimes\cT(\{3,6\},m)$ for $m=1$.}\label{fig:2cross36}
\end{minipage}
\begin{minipage}{.48\textwidth}
	\centering\includegraphics[width=0.9\textwidth]{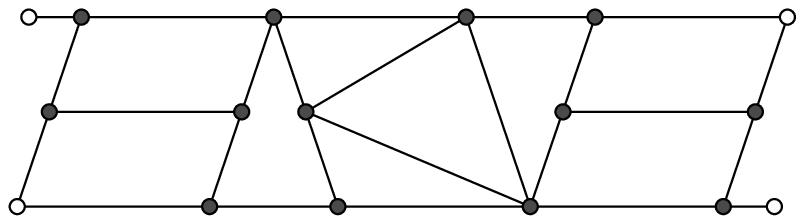}
	\caption{Tile $\otimes\cT(\{3,4,5\},m)$ for $m=1$.}\label{fig:2cross345}
\end{minipage}

\vspace*{3mm}%
\begin{minipage}{.48\textwidth}
	\centering\includegraphics[width=0.9\textwidth]{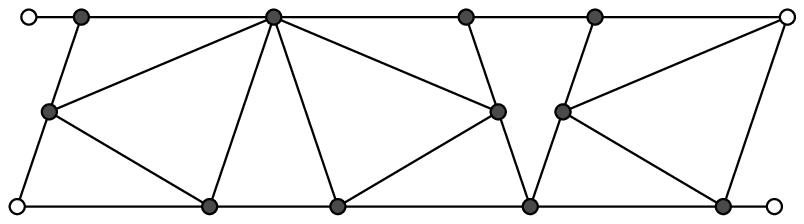}
	\caption{Tile $\otimes\cT(\{3,4,6\},m)$ for $m\!\!=\!\!1$.}\label{fig:2cross346}
\end{minipage}
\begin{minipage}{.48\textwidth}
	\centering\includegraphics[width=0.9\textwidth]{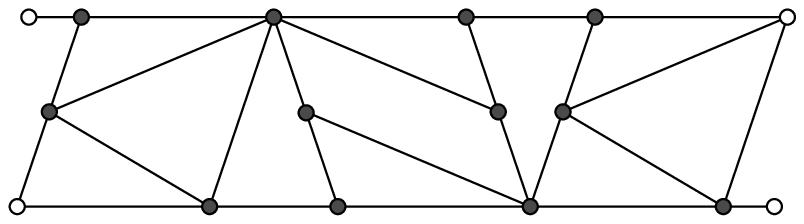}
	\caption{Tile $\otimes\cT(\{3,4,5,6\},m)$ for $m\!\!=\!\!1$.}\label{fig:2cross3456}
\end{minipage}
\end{figure}

For any $\{3\}\subsetneq D \subseteq \{3,4,5,6\}$, $D\not=\{3,5,6\}$,
the constructed simple graph $\circ\big(\cT(D,m)^\updownarrow\big)$ is $3$-connected,
$2$-crossing-critical by Theorem~\ref{th:2critchar},
and contains vertices only with degrees from $D$ (see the pictures),
each at least $m$ times.
Hence $\{\circ\big(\cT (D, m)^\updownarrow\big): m\in \ZZ^+\}$ constitute the
required families, except the case $D=\{3,5,6\}$.
For the latter remaining case, we ``insert''
$^\updownarrow\cT(\{3,6\},m)^\updownarrow$ in place of one
$^{\updownarrow}{T_b}^{\updownarrow}$ in $\cT(\{3,5\},m)$,
providing the following sequence of length $4m+1$:
$$\cT(\{3,5,6\},2m) =(T_a, ^{\updownarrow}{T_b}^{\updownarrow}, T_b,
  \dots, T_b, ^{\updownarrow}{T_b}^{\updownarrow},\> T_a,
  ^{\updownarrow}{T_b}^{\updownarrow}, T_a, \dots,
  ^{\updownarrow}{T_b}^{\updownarrow}, T_a)
$$
Since the same now holds also for
$\circ\big(\cT(\{3,5,6\},2m)^\updownarrow\big)$,
the construction part is finished.

We do similar for case \ref{it:b_general}). These graphs only have degrees 3, 4, 5, and 6, so $D\subseteq \{3,4,5,6\}$.
On the other hand, any join of tiles has at least two vertices of different 
degrees and at least one of them being $3$ or $4$,
so $|D|\ge 2$ and $D\cap\{3,4\}\neq \emptyset$.

For the converse, we must construct a family $\cF$ for any such prescribed set~$D$.
Using Theorem \ref{th:2crossD} \ref{it:a_simple}) (for $D$ such that $3\in D$), 
only the cases $D\cap\{3,4\}=\{4\}$ remain to be resolved.
Define the following sequences of length $2m+1$ each:
$$
\aligned
\cT(\{4,5\},m)&=(T_f, ^{\updownarrow}{T_f}^{\updownarrow}, T_f, \dots, ^{\updownarrow}{T_f}^{\updownarrow}, T_f)\quad \text{(Figure \ref{fig:2cross45})},\\
\cT(\{4,6\},m)&=(T_e, ^{\updownarrow}{T_e}^{\updownarrow}, T_e, \dots, ^{\updownarrow}{T_e}^{\updownarrow}, T_e)\quad \text{(Figure \ref{fig:2cross46})},\\
\cT(\{4,5,6\},m)&=(T_e, ^{\updownarrow}{T_f}^{\updownarrow}, T_e, \dots, ^{\updownarrow}{T_f}^{\updownarrow}, T_e)\quad \text{(Figure \ref{fig:2cross456})},
\endaligned
$$
\begin{figure}[h]
\psfrag{Lww}{$L_{ww}$}
\psfrag{u1}{$\sigma$}
\psfrag{wl}{$\alpha$}
\centering
\begin{minipage}{.48\textwidth}
	\centering\includegraphics[width=0.99\textwidth]{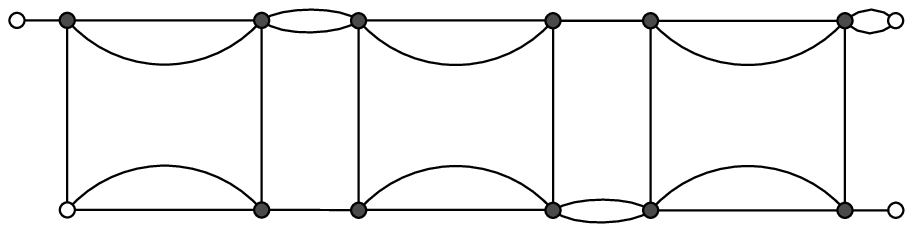}
	\caption{Tile $\otimes\cT(\{4,5\},m)$ for $m=1$.}\label{fig:2cross45}
\end{minipage}
\end{figure}
\begin{figure}[h]
\psfrag{Lww}{$L_{ww}$}
\psfrag{u1}{$\sigma$}
\psfrag{wl}{$\alpha$}
\centering
\begin{minipage}{.48\textwidth}
	\centering\includegraphics[width=0.9\textwidth]{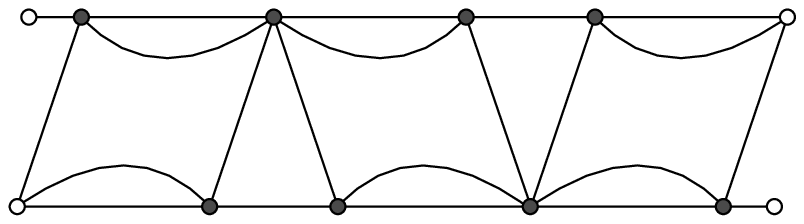}
	\caption{Tile $\otimes\cT(\{4,6\},m)$ for $m=1$.}\label{fig:2cross46}
\end{minipage}
\begin{minipage}{.48\textwidth}
	\centering\includegraphics[width=0.9\textwidth]{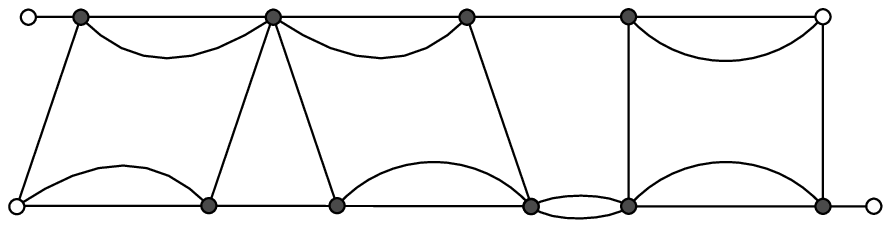}
	\caption{Tile $\otimes\cT(\{4,5,6\},m)$ for $m=1$.}\label{fig:2cross456}
\end{minipage}
\end{figure}
Similarly as above,
for any $D \subseteq \{3,4,5,6\}$, $|D|\ge 2$, $D\cap\{3,4\}=\{4\}$,
the constructed graph $\circ\big(\cT(D,m)^\updownarrow\big)$ is
$3$-connected and $2$-crossing-critical,
contains vertices only with degrees from $D$, each at least $m$ times,
and so we are done.
\qed\end{proof}

\section{Average Degrees in $2$-crossing-critical Families} 
\label{sc:2ccdeg}

In order to study average degrees in $2$-crossing-critical graphs, 
we introduce the \DEF{density characteristics} of a tile $T$ 
as the pair of integers $(a,b)$, where (i) $a$ is the number of vertices 
of $T$, counting wall vertices of degree greater than $1$ as $1/2$
and those of degree $1$ as 0, and (ii) $b$ is the sum of degrees of $T$
counting all degrees in full except those of wall vertices with degree $1$.
Then we define the \DEF{density} of a tile $T$ as $\tfrac{b}{a}$. 
Relevance of this concept is revealed through the following lemma:

\begin{lemma}
\label{lm:characteristics}
Let $T\in\cT_o(\cS)$ be the join of 
$\cT=(T_0, ^{\updownarrow}{T_1}^{\updownarrow}, 
T_2, \dots, ^{\updownarrow}{T_{2m-1}}^{\updownarrow}, T_{2m})$
so that $m\geq 1$ and $T_i \in \cS$ for each $i = 0, 1, 2, \dots, 2m$, 
let $G=\circ (\cT^{\updownarrow})$,  
and let $(a_i,b_i)$ be the density characteristics of the tile $T_i$. 
Then the average degree of $G$ is equal to 
$\left(\sum_{i=0}^{2m} b_i\right)\big/\left(\sum_{i=0}^{2m} a_i\right)$. 
Hence if each $T_i$ has density $\tfrac{b}{a}$, 
then the average degree of $G$ is equal to their density. 
\end{lemma}

\begin{proof}
The wall vertices of adjacent tiles are identified producing the 
cyclization $G$ and possible resulting degree two vertices are suppressed.
Note that any degree 1 wall vertex (of a tile from $\cS$) is always identified with 
another degree 1 wall vertex, resulting in suppression of a degree 2 vertex
in the final graph~$G$. 
Hence the total number of vertices in $G$ equals $\sum_{i=0}^{2m} a_i$ 
and the sum of degrees
equals $\sum_{i=0}^{2m} b_i$, implying the claim.
\qed\end{proof}

Lemma \ref{lm:characteristics} implies that the lowest and highest
achievable average degrees in an infinite family of 
$2$-crossing-critical graphs are determined 
by the lowest and highest density of 
some tiles in $S$. This implies the following theorem:
\begin{theorem}
\label{th:2cross1simple}
A simple, $3$-connected, 
$2$-crossing-critical infinite family of graphs with average 
degree $r\in \QQ$ exists if and only if $r\in\left[3\tfrac{1}{5},4\right]$.
\end{theorem}
\begin{proof}
An elementary checking yields that the smallest density of a simple tile in $T$
is $3\tfrac{1}{5}$, achieved by $T_a$ of characteristics $(5,16)$, and the 
largest is $4$, achieved by $T_c$ and $T_d$ of characteristics $(4,16)$.
Lemma \ref{lm:characteristics} combined with Theorem \ref{th:2critchar} implies
that the average degree $r$ of an infinite family must be in the specified interval.
Sequences consisting of just one of these tiles establish the boundary cases
$r\in\left\{3\tfrac{1}{5},4\right\},$ which are thus easily achievable. 

Let $r=\tfrac{p}{q}\in(3\tfrac{1}{5},4)$. Since sequences consisting of only
$T_a$ and $T_c$, $T_d$ may violate the parity condition 
trying to establish average degree $r$, we add also some tiles $T_b$ 
with characteristics $(5,18)$ to the construction.
Solving diophantine equations, we find a solution
$$\cT (k)= (T_a, {^\updownarrow}{T_a}^{\updownarrow}, 
\dots, T_a, {^\updownarrow}{T_a}^{\updownarrow}, T_b, {^\updownarrow}{T_b}^{\updownarrow}, 
\dots, T_b, {^\updownarrow}{T_b}^{\updownarrow}, T_c, {^\updownarrow}{T_c}^{\updownarrow}, 
\dots, {^\updownarrow}{T_c}^{\updownarrow}, T_c)
,$$ 
where tiles
$T_a$ and ${^\updownarrow}{T_a}^{\updownarrow}$ together appear $\ell_k=(96q-24p-4)(2k-1)$-times, 
$T_b$~and~${^\updownarrow}{T_b}^{\updownarrow}$ together appear $m_k=8(2k-1)$-times,
and $T_c$, ${^\updownarrow}{T_c}^{\updownarrow}$ together appear $n_k=(30p-96q-5)(2k-1)$-times.

Since $3\tfrac{1}{5}<r<4$, we have $\ell_k,n_k>0$ for sufficiently
large $p,q$ (which are not required to be relatively prime) such
that~$r=\tfrac{p}{q}$.
The total length of the sequence $\cT (k)$ then is $(6p-1)(2k-1)$,
which is an odd number required for $\circ\cT(k)^\updownarrow\in\cG(\cS)$. 
Now Lemma \ref{lm:characteristics} and a routine calculation imply 
that $\cF_r=\dset{\circ\cT(k)^\updownarrow}{k\in\NN}$ is a family of 
simple, 3-connected, 2-crossing-critical graphs with average degree 
$\tfrac{96p(2k-1)}{96q(2k-1)}=\tfrac{p}{q}=r$.
\qed\end{proof}

\begin{theorem}
\label{th:2cross1}
A $3$-connected, $2$-crossing-critical infinite family with average degree $r\in \QQ$ exists 
if and only if $r\in\left[3\tfrac{1}{5},4\tfrac{2}{3}\right]$.
\end{theorem}
\begin{proof}
The argument is as in the previous proof, with an additional observation
that the largest possible density $4\tfrac{2}{3}$ among non-simple tiles in $\cS$ 
is obtained by the tile $T_e$ of characteristics $(3,14)$. 
For $r=\tfrac{p}{q}\in\left(3\tfrac{1}{5},4\tfrac{2}{3}\right)$, we use the sequence 
$$\cT (k)= (T_a, {^\updownarrow}{T_a}^{\updownarrow}, \dots, T_a, 
{^\updownarrow}{T_a}^{\updownarrow}, T_c, {^\updownarrow}{T_c}^{\updownarrow}, \dots, 
{^\updownarrow}{T_c}^{\updownarrow}, T_c, {^\updownarrow}{T_e}^{\updownarrow}, T_e,  
\dots, {^\updownarrow}{T_e}^{\updownarrow}, T_e),$$ 
where $T_a$ and ${^\updownarrow}{T_a}^{\updownarrow}$ together
appear $(112q-24p-4)(2k-1)$-times,
$T_c$ and ${^\updownarrow}{T_c}^{\updownarrow}$ together
appear $11(2k-1)$-times, and
$T_e$ and ${^\updownarrow}{T_e}^{\updownarrow}$ together
appear $(40p-128q-8)(2k-1)$-times.
The length of this sequence is $(16p-16q-1)(2k-1)$
which is an odd number, 
and the average degree of each graph $\circ(\cT(k)^\updownarrow)$ is $r$.
\qed\end{proof}

The technique from the proofs of Theorems~\ref{th:2cross1simple}
and~\ref{th:2cross1} can be generalized to 
the following auxiliary claim.
Recall the sets $\cT_o(\cS),\cT_e(\cS)$
of odd and even tiles from Definition~\ref{df:2critchar}
and note that every tile of $\cT_o(\cS)\cup\cT_e(\cS)$ 
has precisely one wall vertex of degree~$1$ in each wall, 
and that such vertices are identified and contracted 
when joining tiles.

\begin{lemma}
\label{cr:Dinterval}
Let $T_1,T_2\in\cT_o(\cS)\cup\cT_e(\cS)$ be two tiles,
let the density characteristics of $T_i$, $i=1,2$ be $(a_i,b_i)$
and $\tfrac{b_1}{a_1} \leq \tfrac{b_2}{a_2}$.
Let $D_i$ be the set of degrees of the non-wall vertices of 
$T_i$, let $d_i^l$, $d_i^r$ be the degrees of its left and right 
wall vertices that are not~$1$, and moreover let
$\bar D_i:=D_i\cup\{d_i^l+d_i^r\}$ and $D_{12}:=\{d_1^r+d_2^l,d_1^l+d_2^r\}$. 

Assume $r\in\left[\tfrac{b_1}{a_1},\tfrac{b_2}{a_2}\right]$ 
is a rational number.
Then there exists a $D$-max-uni\-versal family $\cF(D,r)$ of\/ $3$-connected
$2$-crossing-critical graphs from $\cG(\cS)$ with average degree $r$, 
whenever (at least) one of the conditions (i)--(\ref{it:Dlast}) is satisfied:
\begin{enumerate}[i)]
\item $r=\tfrac{b_i}{a_i}$ for some $i\in\{1,2\}$ and $D=\bar D_i$ holds;
\item $r=\tfrac{b_1}{a_1}=\tfrac{b_2}{a_2}$ and
  $D$ is one of $\bar D_1\cup \bar D_2$, $D_1\cup D_2\cup D_{12}$,
  $D_1\cup \bar D_2\cup D_{12}$, $\bar D_1\cup D_2\cup D_{12}$
  or $\bar D_1\cup \bar D_2\cup D_{12}$;
\item $r\in\left(\tfrac{b_1}{a_1},\tfrac{b_2}{a_2}\right)$
 and one of $D=\bar D_1\cup\bar D_2$ or                 
  $D=\bar D_1\cup \bar D_2\cup D_{12}$ holds;
\item $r\in\left(\tfrac{b_1}{a_1},\tfrac{b_2}{a_2}\right)$,
  and one of the following special subcases holds:\\
  $r<\tfrac{b_1+b_2}{a_1+a_2}$ and $D=\bar D_1\cup D_2\cup D_{12}$, or
  $r=\tfrac{b_1+b_2}{a_1+a_2}$ and $D= D_1\cup D_2\cup D_{12}$, or
  $r>\tfrac{b_1+b_2}{a_1+a_2}$ and $D= D_1\cup\bar D_2\cup D_{12}$.
\label{it:Dlast}
\end{enumerate}
Moreover, if both $T_1$ and $T_2$ are simple graphs, then also the
graphs in $\cF(D,r)$ are simple.
\end{lemma}

\begin{proof}
To simplify the arguments of the coming proof, we resolve beforehand the necessary 
composition and parity issues with respect to our use of Theorem~\ref{th:2critchar}.
That is, for an integer $m\geq3$, assume we have already defined a sequence 
$\cU=(U_1,U_2,\dots,U_m)\in \{T_1,T_2\}^m$ from the given tiles,
such that $T_1$ occurs $\ell$ times in~$\cU$ and
$r=\tfrac{\ell b_1+(m-\ell)b_2}{\ell a_1+(m-\ell)a_2}$.
Then we claim that there exists $G_m \in \cG(\cS)$ of average degree exactly~$r$,
such that $G_m$ is composed of the tiles of $\cU$ or their inverses in the
given order and of possibly one additional tile $U_0\in\cT_o(\cS)$.
Moreover, this additional tile $U_0$ is independent of $m$ (and hence it will 
not affect $D$-max-universality for any~$D$).

Indeed, we can define $\cU'_1:=(U_1)$ and for $i\geq1$ inductively
$\cU'_{i+1}:=(\cU'_i,\>U_{i+1})$ if $(\otimes\, \cU'_i) \in\cT_e(\cS)$
and $\cU'_{i+1}:=(\cU'_i,\,{}^\updownarrow U_{i+1}{}^\updownarrow)$ otherwise.
If $\otimes \cU'_m$ is odd, then we finish by picking the cyclization
$G_m:=\circ(\cU'_m{}^\updownarrow)$, which satisfies the claimed properties
by Theorem~\ref{th:2critchar} and Lemma~\ref{lm:characteristics}.
If $\otimes \cU'_m$ is even, then let $U_0\in\cT_o(\cS)$ be any fixed tile such that
$\circ(U_0{}^\updownarrow)$ is a $2$-crossing-critical graph with average degree
$r$, which exists by Theorem \ref{th:2cross1}.
We now analogously finish by setting
$G_m:=\circ(\cU'_m,\,U_0{}^\updownarrow)$.
Note that $G_m$ would be simple if both $T_1,T_2$ are simple, since then we
could choose also $U_0$ as simple by Theorem~\ref{th:2cross1simple}.

Consequently, assuming that we can define such a sequence
$\cU\in \{T_1,T_2\}^m$ as above for all $m\in M$ of some infinite index set~$M$,
the set $\cF(D,r):=\cG_r(\cU)=\{G_m:m\in M\}$
would be the claimed infinite family of $3$-connected
$2$-crossing-critical graphs from $\cG(\cS)$ with average degree $r$,
where $D$ depends on particular $\cU$ and is subject to 
further case analysis  which will conclude the proof.

\medskip
Assume (i); $D=\bar D_i$.
We simply set $\cU=(T_i,T_i,\dots,T_i)$ to be a list of 
length at least three and finish by the previous.

Assume (ii).
If $D=\bar D_1\cup \bar D_2$, then it suffices to set
$\cU=(T_1,T_1\dots,T_1,$ $T_2,T_2\dots,T_2)$.
On the other hand, if $D=D_1\cup D_2\cup D_{12}$, then we set 
$\cU'=(T_1,T_2,T_1,T_2,\dots,T_1,T_2)$.
The remaining three cases of (ii) follow from the latter sequence $\cU'$ simply by replacing
each $T_2$ with $T_2,T_2$ (duplicating), or each $T_1$ with $T_1,T_1$, or both.

Assume (iii) where $r=\frac pq$, and choose any $\ell\in\NN$.
Define $k_1=(qb_2-pa_2)$, $k_2=(pa_1-qb_1)$, and $m=(k_1+k_2)\ell$. 
Note that $r\in\left(\tfrac{b_1}{a_1},\tfrac{b_2}{a_2}\right)$
implies $k_1,k_2>0$.
For $D=\bar D_1\cup\bar D_2$, take $\cU= (T_1, T_1,\dots, T_1, 
T_1,T_2, T_2, \dots,T_2, T_2)\in\{T_1,T_2\}^m$ where
$T_i$ appears $k_i\ell$ times for $i=1,2$. 
A routine calculation verifies that the family $\cG_r(\cU)$ 
for $M=\dset{(k_1+k_2)j}{j\in\NN}$ is as desired, with the average degree equal to~$r$. 
For $D=\bar D_1\cup\bar D_2\cup D_{12}$, reorder $\cU$ into 
$\cU':=(T_1,T_1,T_2,T_2,T_1,T_1,T_2,T_2, \dots,
	T_i,T_i,\dots,T_i)\in\{T_1,T_2\}^m$
where $i\in\{1,2\}$ is such that $k_i>k_{3-i}$, or $i=2$ if $k_1=k_2$,
and take the family $\cG_r(\cU')$ instead.

Assume (iv), and consider $k_1$ and $k_2$ as in (iii).
Note that $r<\tfrac{b_1+b_2}{a_1+a_2}$,
  $r=\tfrac{b_1+b_2}{a_1+a_2}$,
  $r>\tfrac{b_1+b_2}{a_1+a_2}$
is, in order, equivalent to $k_1>k_2$, $k_1=k_2$, $k_1<k_2$.
In this order, we can hence choose the following sequences
as reorderings of previous $\cU$ such that the claimed three cases will be
covered:
$\cU^1= (T_1,T_2,T_1,T_2,\dots,T_1,T_2,T_1,T_1,\dots,T_1)\in\{T_1,T_2\}^m$
for $k_1>k_2$,~
$\cU^2= (T_1,T_2,T_1,T_2,\dots,T_1,T_2,T_2,\dots,T_2,T_2)\in\{T_1,T_2\}^m$
for $k_2>k_1$, and
$\cU^0= (T_1,T_2,T_1,T_2,\dots,T_1,T_2,T_1,T_2)\in\{T_1,T_2\}^m$
for $k_1=k_2$.
\qed
\end{proof}

We finish with the last theorem which summarizes results of this and previous section.
\begin{theorem}
\label{th:2ccDmu}
Let $D$ be such that there exists a $D$-max-universal $3$-connected 
$2$-crossing-critical family. Then let $I_{D}$ (or $I_{D}^s$ for simple graphs) 
be the set of all rational numbers, such that there is a $D$-max-universal $3$-connected 
$2$-crossing-critical (simple) family with average degree $r$ if and only if $r\in I_D$ 
($r\in I_D^s$). Then $I_{D}^s$ and $I_D$ are rational intervals and moreover:
\begin{center}
\vspace*{3mm}%
\setlength\extrarowheight{2pt}
   \begin{tabular}{| c || c | c |}
     \hline
     $D$ & $I_{D}^s$ & $I_{D}$ \\ \hline \hline
     $\{3, 4\}$ & ~$[\frac {16}5, \frac {18}5]$~ & ~$[\frac {16}5, \frac{15}4]$~ \\ \hline
     $\{3, 5\}$ & $\{\frac {17}5\}$ & $[\frac {17}5, \frac{11}3]$ \\ \hline
     $\{3, 6\}$ & $\{\frac {18}5\}$ & $\{\frac {18}5\}$ \\ \hline
     $\{4, 5\}$ & $\emptyset$ & $\{\frac 92\}$ \\ \hline
     $\{4, 6\}$ & $\emptyset$ & $\{\frac {14}3\}$ \\ \hline
     $\{3, 4, 5\}$ & $(\frac {16}5, 4]$ & $(\frac {16}5, \frac 92)$ \\ \hline
     $\{3, 4, 6\}$ & $(\frac {16}5, 4]$ & $(\frac {16}5, \frac {14}3)$ \\ \hline
     $\{3, 5, 6\}$ & $(\frac {17}5, \frac {18}5)$ & $(\frac {17}5, \frac{11}3)$ \\ \hline
     $\{4, 5, 6\}$ & $\emptyset$ & $(\frac 92, \frac {14}3)$ \\ \hline
     $\{3, 4, 5, 6\}$ & $(\frac {16}5, 4]$ & $(\frac {16}5, \frac {14}3)$ \\ \hline
   \end{tabular}
\end{center}
\end{theorem}

\vspace*{3mm}%
\begin{proof}
The list of possible degree sets $D$ was established by Theorem 
Theorem~\ref{th:2crossD}\,\ref{it:a_simple}) for simple 2-crossing-critical graphs 
and by Theorem~\ref{th:2crossD}\,\ref{it:b_general}) for general $2$-crossing-critical 
graphs. We proceed as follows: for each such set $D$, we examine the 
set of tiles $\cS$. If a tile $T$ has its
degrees in $D$, then $T$ is a candidate for the construction. 
Among all candidates, we will focus on the ones
with the smallest and with the largest density.
Lemma \ref{lm:characteristics} implies that for $r$ 
outside of the open (closed) interval $J_D$ ($\overline{J_D}$) 
defined by these min/max densities, a $D$-max-universal 
family of $2$-crossing-critical graphs 
with prescribed average degree does not exist.

On the other hand, if the chosen two tiles and their joins
produce all the degrees in $D$, Corollary~\ref{cr:Dinterval}
immediately yields a $D$-max-universal $2$-crossing-critical family
with prescribed average degree $r$ for any $r\in J_D$. 
Though, if $r\in\overline{J_D}\setminus J_D$ is one of the interval ends,
special care has to be taken since only tiles with the highest 
(lowest, respectively) density themselves need to produce all the suitable degrees.
As we will see, the latter is not always possible, especially
if there is only one tile with the highest (lowest) density.

The rest of the proof relies on a mostly simple case checking.
Obtaining a list of candidate tiles is routine, as well as 
finding the one(s) with smallest and largest density. 
See below for selected details.
For each possible set $D$, we record the corresponding min- and max-tiles 
$T_1$ and $T_2$ which we use
(cf.~Figure~\ref{fig:2crossexamples}) in the following table:

\begin{center}
\setlength\extrarowheight{3pt}
   \begin{tabular}{| c || c |c |c || c |c |c |c  ||}
     \hline
     $D$ & $I_{D}^s$ & $T_1$ & $T_2$  & $I_{D}$ & $T_1$ & $T_2$ \\ \hline \hline

     $\{3, 4\}$ & ~$[\frac {16}5, \frac {18}5]$~ & ~~$T_a$~~ & $T_g$ & $[\frac {16}5, \frac{15}4]$ & $T_a$ & $T_n$ \\ \hline

     $\{3, 5\}$ &  $\{\frac {17}5\}$ & ~$T_a$~ & $T_b$& $[\frac {17}5, \frac{11}3]$ &
     ~$T_b\otimes{}^\updownarrow T_a^\updownarrow$~ & $T_p$ \\ \hline

     $\{3, 6\}$ & $\{\frac {18}5\}$ & $T_b$ & $T_b$ & $\{\frac {18}5\}$ & $T_b$ & $T_b$ \\ \hline

     $\{4, 5\}$ & $\emptyset$ & & & $\{\frac 92\}$ & $T_f$ & $T_f$ \\ \hline

     $\{4, 6\}$ & $\emptyset$ & & & $\{\frac {14}3\}$ & $T_e$ & $T_e$  \\ \hline

     $\{3, 4, 5\}$  & $(\frac {16}5, 4]$ & $T_a$ & $T_d$ & 
     $(\frac {16}5, \frac 92)$ & $T_a$ & $T_f$ \\ \hline

     $\{3, 4, 6\}$ & $(\frac {16}5, 4]$ & $T_a$ & 
     $T_c\otimes{}^\updownarrow T_d^\updownarrow$ & $(\frac {16}5, \frac {14}3)$ & $T_a$ & $T_e$ \\ \hline

     $\{3, 5, 6\}$ & $(\frac {17}5, \frac {18}5)$ & $T_a$ & $T_b$ 
	& $(\frac {17}5, \frac{11}3)$ & $T_b\otimes{}^\updownarrow T_a^\updownarrow$  & $T_p$ \\ \hline

     $\{4, 5, 6\}$ & $\emptyset$ & & & $(\frac 92, \frac {14}3)$& $T_f$ & $T_e$ \\ \hline
     
     $\{3, 4, 5, 6\}$ & $(\frac {16}5, 4]$ & $T_a$ & \mbox{\small$T_c\otimes{}^\updownarrow T_d^\updownarrow\otimes T_d$} &
     ~$(\frac {16}5, \frac {14}3)$~ &     $T_a$ & ~~$T_e$~~ \\ \hline
   \end{tabular}
\end{center}
\vspace*{3mm}%

We first consider simple graphs (the left-hand side of the table).
For instance, if $D=\{3,4\}$,
then the candidate tiles must not contain wall vertices of degree greater
than $2$ (since those would produce vertices of degree $5$ or higher,
no matter what are other tiles in the sequence),
and so we only have $T_a$ and $T_g$ with densities $\frac{16}5$ and
$\frac{18}5$, respectively.
In this case, the whole closed interval $[\frac {16}5, \frac {18}5]$
can be achieved by Lemma~\ref{cr:Dinterval} (iii) and (i).
On the other hand, in subsequent cases $D\supsetneq\{3,4\}$,
the lowest average degree $\frac {16}5$ can never be achieved
since $T_a$ is the only tile of density $\frac{16}5$ and by itself 
it cannot produce degrees $5$ or $6$.
All the cases $D\supsetneq\{3,4\}$ thus lead to the interval
$(\frac {16}5, 4]$, where the upper bound follows from
Theorem~\ref{th:2cross1simple} and is achieved by suitable combinations
of tiles $T_c$ and~$T_d$.

If $D=\{3,5\}$,
then the candidate tiles are $T_a$ and $T_b$ (other tiles contain 
a degree-$4$ vertex) of densities $\frac{16}5$ and $\frac{18}5$.
Though, since degrees $4$ and $6$ cannot occur frequently in the graphs,
it is not possible to combine $T_a$ with itself or $T_b$ with itself.
Hence the only admissible average degree is
$\frac {17}5=\frac {16+18}{5+5}$ obtained by the equality subcase of
Lemma~\ref{cr:Dinterval} (iv).
The case of $D=\{3,5,6\}$ is similar in that we cannot combine $T_a$ with
itself, and even the lower bound of $\frac {17}5$ cannot be achieved since
we need to frequently combine $T_b$ with itself to produce degree~$6$.
Likewise, the max average degree $\frac {18}5$ cannot be achieved
for $D=\{3,5,6\}$ since joining $T_b$ only with itself does not produce degree~$5$.
All values in $(\frac {17}5, \frac {18}5)$ get produced by
Lemma~\ref{cr:Dinterval} (iv).
The remaining case $D=\{3,6\}$ follows easily.

\medskip
Consider now also non-simple graphs (the right-hand side of the table).
This case brings a new possibility that $3\not\in D$, which we resolve first.
The assumption $3\not\in D$ leaves only two candidate tiles $T_e$ and 
$T_f$ of densities $\frac{14}3$ and $\frac92$
(where $T_e$ is at the same time the overall unique ``densest'' tile on our list).
Combining $T_e$ with itself produces exclusively degrees $4$ and $6$,
hence the result for $D=\{4,6\}$.
We similarly get the unique solution for $D=\{4,5\}$ with~$T_f$.
In order to produce all degrees $D=\{4,5,6\}$,
we have to also combine $T_e$ with $T_f$, and so the solution then
is the open interval $(\frac92, \frac {14}3)$
by Lemma~\ref{cr:Dinterval} (iii).

For $3\in D$, the solution intervals $I_D$ obviously contain the
corresponding simple-graph solutions $I_D^s$ and, actually,
a routine case analysis shows that the lower bounds
(left ends of the intervals) are always the same as for simple graphs.
We in particular emphasise that for $D=\{3,5\}$ and $D=\{3,5,6\}$,
the previous (simple) construction by Lemma~\ref{cr:Dinterval} (iv)
has actually defined the combined tile
$T_b\otimes{}^\updownarrow T_a^\updownarrow$
which we will now use as the new min-tile in further construction.

Concerning the upper bounds of the intervals $I_D$ (in this non-simple case),
we easily check that the highest density $\frac{14}3$ of all is achieved
only by $T_e$, which can by itself produce only $D=\{4,6\}$ (as mentioned above).
In the other cases of $D\supsetneq\{4,6\}$, the right end $\frac{14}3$ 
of $I_D$ thus must be open and the rest follows by Lemma~\ref{cr:Dinterval} (iii).
If $4\not\in D$, then the max-density candidates are $T_p$ of $\frac{11}3$ 
or $T_b$ of $\frac{18}5$ (where $\frac{18}5<\frac{11}3$), 
depending on whether $5\in D$ since $T_p$ contains a degree-$5$ vertex.
The rest is mostly similar, except the case of $D=\{3,5,6\}$.
In the latter case, since $T_p$ combined with itself does not produce degree~$6$,
the right end of $I_D$ is open for $D=\{3,5,6\}$.
Moreover, we cannot simply cover the whole interval
$I_D=(\frac{17}5,\frac{11}3)$ in one shot (we would miss degree~$6$),
and so we make a union of the simple case together with the interval
$(\frac{52}{15},\frac{11}3)$ which is covered by Lemma~\ref{cr:Dinterval}
(iii) for tiles $T_b\otimes{}^\updownarrow T_a^\updownarrow\otimes T_b$
and~$T_p$.

The remaining cases of our discussion of non-simple upper bounds
are $D=\{3,4\}$ and $D=\{3,4,5\}$.
For $D=\{3,4\}$, an exhaustive search reveals a new
max-density candidate $T_n$ of $\frac{15}4$.
For $D=\{3,4,5\}$, the aforementioned tile $T_e$ is a single max-density
candidate according to our classification.
However, this is a special case since combining $T_e$ with
itself necessarily creates degree-$6$ vertices.
Thus, in order to surpass the density $\frac 92$ of the next 
lower available tile $T_f$, one would need to combine $T_e$
with a tile of wall degrees $2$ and density at least $\frac{13}3$ which
does not exist.
Hence the true upper bound in this case is $\frac 92$ of $T_f$
which, though, cannot be achieved due to~$3\in D$.
Finally, having resolved all lower and upper bounds, we finish again by
Lemma~\ref{cr:Dinterval}~(iii).
\qed
\end{proof}

%%%%%%%%%%%%%%%%%%%%%%%%%%%%%%%%%%%%%%%%%%%%%%%%%%%%%%%%%%%%%%%%%%%%%%%%%%% 

\section{Final Remarks}
\label{sc:final}

We conclude with some challenges for further possible research.
The statement of Theorem~\ref{thm:alluniversal} 
always requires $4\in D$, but from 
Theorem~\ref{th:2crossD} we know that there exist
$D$-max-universal families of simple, 
$3$-connected, $2$-crossing-critical graphs
for $D=\{3,5\}$ and $D=\{3,6\}$ 
(Figures~\ref{fig:2cross35},~\ref{fig:2cross36}),
 e.g., when $4\not\in D$, and these can be generalized to any $k>2$ by 
a zip product with copies of $K_{3,3}$.

Hence it is an interesting open question of whether there exists a
$D$-max-universal $k$-crossing-critical family such that
$D\cap\{3,4\}=\emptyset$.
It is unlikely that the answer would be easy since the question is related
to another long standing open problem---whether there exists a $5$-regular
$k$-crossing-critical infinite family.
Related to this is the same question of existence of a $4$-regular
$k$-crossing-critical family,
which does exist for $k=3$ \cite{cit:rtcrit} and the construction can be
generalized to any $k\geq6$, but the cases $k=4,5$ remain open.

Many more questions can be asked in a direct relation to the statement of
Theorem~\ref{thm:alluniversal-avgdeg}, but we only 
mention a few of the most interesting ones.
E.g., if $6\not\in D$, can the average degree of such a family be from the
interval $[5,6)$?
Or, assuming $3\in D$ but $4\not\in D$, for which sets $D$ one can achieve
$D$-max-universality and what are the related average degrees?

\smallskip
Concerning specifically $2$-crossing-critical graphs, there are no open
questions or cases left by the results of Section~\ref{sc:2cc}.
Yet, there is a natural open question related to Theorem~\ref{th:2ccDmu},
namely; how would the sets of admissible average degrees $I_D$ and $I_D^s$
change if we require all the vertex degrees in the constructed
$D$-max-universal family to belong to~$D$?
There is a two-fold effect of this restriction.
First, we would not be allowed to resolve the parity problem
by adding an arbitrary small tile,
and second, we could get some undesired degrees when joining two different tiles.

This question of precise degree set is nontrivial since, for example,
in the case of simple graphs and $D=\{3,4\}$, we have only two tiles
and there exist values of $r$ which actually force the total number of
these tiles to be even (so our construction is not realizable).
We leave this question open for further investigation.

\smallskip
We finish with another interesting structural conjecture:
\begin{conjecture}
There is a function $g:\NN\to\RR^+$ such that,
any sufficiently large simple $3$-connected $k$-crossing-critical graph has
average degree greater than $3+g(k)$.
\end{conjecture}
Note that corresponding result ``on the upper side'', i.e., bounding average degrees 
away from~$6$, has been established in \cite{cit:nestedcycles}. Furthermore, note that the staircase strip generalization of Kochol's original implies $g(\binom{n}{2}-1)<\tfrac{1}{4n-7}$, cf.\ Theorem~\ref{thm:constrodd}. The following problem therefore poses itself naturally: 

\begin{problem}
Do staircase strips yield the sparsest $k$-crossing-critical graphs, ie.\ does there exist a $k$-crossing-critical family of graphs with average degree less than $$3+\frac{1}{2\sqrt{1+8(k+1)}-5}?$$
\end{problem}

%%%%%%%%%%%%%%%%%%%%%%%%%%%%%%%%%%%%%%%%%%%%%%%%%%%%%%%%%%%%%%%%%%%%%%%%%%%%%%%%%%%%
\begin{small}
\bibliographystyle{abbrv}
\bibliography{phcross}

\begin{thebibliography}{10}

\bibitem{Ajtai19829}
M.~Ajtai, V.~Chv\'atal, M.~Newborn, and E.~Szemer\'edi.
\newblock Crossing-free subgraphs.
\newblock In {\em Theory and Practice of Combinatorics}, volume~60 of {\em
  North-Holland Mathematics Studies}, pages 9 -- 12. North-Holland, 1982.

\bibitem{cit:zipprod}
D.~Bokal.
\newblock On the crossing numbers of cartesian products with paths.
\newblock {\em J. Combin. Theory Ser. {B}}, 97(3):381--384, 2007.

\bibitem{cit:avgcrit}
D.~Bokal.
\newblock Infinite families of crossing-critical graphs with prescribed average
  degree and crossing number.
\newblock {\em J. Graph Theory}, 65(2):139--162, 2010.

\bibitem{cit:zip3}
D.~Bokal, M.~Chimani, and J.~Lea{\~{n}}os.
\newblock Crossing number additivity over edge cuts.
\newblock {\em European J. Combin.}, 34(6):1010--1018, 2013.

\bibitem{cit:2critchar}
D.~Bokal, B.~Oporowski, R.~B. Richter, and G.~Salazar.
\newblock Characterizing 2-crossing-critical graphs.
\newblock {\em Adv. in Appl. Math.}, 74:23--208, 2016.

\bibitem{cit:dvorakmohar}
Z.~Dvo\v{r}\'ak and B.~Mohar.
\newblock Crossing-critical graphs with large maximum degree.
\newblock {\em J. Combin. Theory Ser. {B}}, 100(4):413--417, 2010.

\bibitem{cit:embgrids}
J.~F. Geelen, R.~B. Richter, and G.~Salazar.
\newblock Embedding grids in surfaces.
\newblock {\em European J. Combin.}, 25(6):785--792, 2004.

\bibitem{cit:nestedcycles}
C.~Hern{\'{a}}ndez{-}V{\'{e}}lez, G.~Salazar, and R.~Thomas.
\newblock Nested cycles in large triangulations and crossing-critical graphs.
\newblock {\em J. Combin. Theory Ser. {B}}, 102(1):86--92, 2012.

\bibitem{cit:pathcrit}
P.~Hlin\v{e}n{\'{y}}.
\newblock Crossing-critical graphs and path-width.
\newblock In P.~Mutzel, M.~J{\"{u}}nger, and S.~Leipert, editors, {\em Graph
  Drawing, 9th International Symposium, {GD} 2001, Revised Papers}, volume LNCS
  2265 of {\em Lecture Notes in Comput. Sci.}, pages 102--114. Springer, 2002.

\bibitem{cit:pw}
P.~Hlin\v{e}n{\'{y}}.
\newblock Crossing-number critical graphs have bounded path-width.
\newblock {\em J. Combin. Theory Ser. {B}}, 88(2):347--367, 2003.

\bibitem{cit:newinfcrit}
P.~Hlin\v{e}n{\'{y}}.
\newblock New infinite families of almost-planar crossing-critical graphs.
\newblock {\em Electron. J. Combin.}, 15(1), 2008.

\bibitem{cit:starsbonds}
P.~Hlin\v{e}n{\'{y}} and G.~Salazar.
\newblock Stars and bonds in crossing-critical graphs.
\newblock {\em J. Graph Theory}, 65(3):198--215, 2010.

\bibitem{cit:kochol}
M.~Kochol.
\newblock Construction of crossing-critical graphs.
\newblock {\em Discrete Math.}, 66(3):311--313, 1987.

\bibitem{cit:leighton}
T.~Leighton.
\newblock {\em Complexity Issues in {VLSI}}.
\newblock Foundations of Computing Series. {MIT} {P}ress, {C}ambridge, {MA},
  1983.

\bibitem{cit:pltile}
B.~Pinontoan and R.~B. Richter.
\newblock Crossing numbers of sequences of graphs {II:} {P}lanar tiles.
\newblock {\em J. Graph Theory}, 42(4):332--341, 2003.

\bibitem{cit:gentile}
B.~Pinontoan and R.~B. Richter.
\newblock Crossing numbers of sequences of graphs {I}: {G}eneral tiles.
\newblock {\em Australas. J. Combin.}, 30:197--206, 2004.

\bibitem{cit:rtcrit}
R.~B. Richter and C.~Thomassen.
\newblock Minimal graphs with crossing number at least \emph{k}.
\newblock {\em J. Combin. Theory Ser. {B}}, 58(2):217--224, 1993.

\bibitem{cit:avgcr4}
G.~Salazar.
\newblock Infinite families of crossing-critical graphs with given average
  degree.
\newblock {\em Discrete Math.}, 271(1-3):343--350, 2003.

\bibitem{cit:siran}
J.~\v{S}ir{\'{a}}\v{n}.
\newblock Infinite families of crossing-critical graphs with a given crossing
  number.
\newblock {\em Discrete Math.}, 48(1):129--132, 1984.

\end{thebibliography}
\end{small}

%%%%%%%%%%%%%%%%%%%%%%%%%%%%%%%%%%%%%%%%%%%%%%%%%%%%%%%%%%%%%%%%%%%%%%%%%%%%
%%%%%%%%%%%%%%%%%%%%%%%%%%%%%%%%%%%%%%%%%%%%%%%%%%%%%%%%%%%%%%%%%%%%%%%%%%%%
%%%%%%%%%%%%%%%%%%%%%%%%%%%%%%%%%%%%%%%%%%%%%%%%%%%%%%%%%%%%%%%%%%%%%%%%%%%%
%%%%%%%%%%%%%%%%%%%%%%%%%%%%%%%%%%%%%%%%%%%%%%%%%%%%%%%%%%%%%%%%%%%%%%%%%%%%

\end{document}